\documentclass[10pt, a4paper, leqno]{article}

\usepackage[english]{babel}
\usepackage[utf8]{inputenc}
\usepackage{amsmath}
\usepackage{graphicx}
\usepackage[colorinlistoftodos]{todonotes}
\usepackage{amssymb,amscd,latexsym,epsfig,xy}
\usepackage[mathscr]{euscript}
\usepackage{amsthm}
\usepackage{tikz}
\usepackage{MnSymbol}
\usetikzlibrary{matrix,arrows,decorations.pathmorphing}
\usepackage{hyperref}
\usepackage{enumerate}
\usepackage{float}
\usepackage{tikz-cd}
\usepackage{relsize}
\usepackage{mathtools}

\numberwithin{equation}{section}

\newtheorem{thm}{Theorem}[section]
\newtheorem{defn}[thm]{Definition}
\newtheorem{lem}[thm]{Lemma}
\newtheorem{cor}[thm]{Corollary}

\newtheorem{prop}[thm]{Proposition}
\newtheorem{ex}[thm]{Example}

\newcommand\A{\mathbb A}
\newcommand\C{\mathbb C}
\newcommand\E{\mathbb E}
\newcommand\F{\mathbb F}

\newcommand\PP{\mathbb P}
\newcommand\R{\mathbb R}
\newcommand\Q{\mathbb Q}
\newcommand\Z{\mathbb Z}
\newcommand\cC{\mathcal C}
\newcommand\cG{\mathcal G}
\newcommand\cO{\mathcal O}

\newcommand\cX{\mathcal X}
\newcommand\fD{\mathfrak D}

\newcommand\vdim{\operatorname{vdim}}
\newcommand\ev{\operatorname{ev}}
\newcommand\pr{\operatorname{pr}}
\newcommand\Nmax{\operatorname{Nmax}}

\newcommand\virt{\mathrm{virt}}

\newcommand\rel{\mathrm{rel}}
\newcommand\floor{\mathrm{floor}}
\newcommand\loga{\mathrm{log}}
\newcommand\Aut{\mathrm{Aut}}
\newcommand\cM{\mathcal{M}}

\title{Refined floor diagrams from higher genera and lambda classes}

\author{Pierrick Bousseau}

\date{}

\begin{document}

\maketitle

\begin{abstract}
We show that, after the change of variables $q=e^{iu}$, refined floor diagrams for 
$\PP^2$ and Hirzebruch surfaces compute generating series of higher genus relative Gromov-Witten 
invariants with insertion of a lambda class.
The proof uses an inductive application of the degeneration formula in relative Gromov-Witten theory and an explicit result in relative Gromov-Witten theory of $\PP^1$.

Combining this result with the similar looking refined tropical
correspondence theorem for log Gromov-Witten invariants, 
we obtain a non-trivial relation between relative and 
log Gromov-Witten invariants for 
$\PP^2$ and Hirzebruch surfaces.
We also prove that the Block-G\"ottsche invariants of $\F_0$ and $\F_2$ are related by the Abramovich-Bertram formula.



\end{abstract}

\setcounter{tocdepth}{1}

\tableofcontents

\thispagestyle{empty}

\section{Introduction}

\subsection{Overview}
Floor diagrams, introduced by Brugall\'e and Mikhalkin
\cite{MR2359091} \cite{MR2500574}, are combinatorial objects that are used to provide a solution to enumerative problems concerning real and complex curves in $h$-transverse toric surfaces. Particular examples of $h$-transverse toric surfaces are the projective plane $\PP^2$ and Hirzebruch surfaces.

One way to understand the relation between 
floor diagrams and curve counting is based on tropical geometry. Mikhalkin's correspondence theorem \cite{MR2137980} relates tropical curves in $\R^2$ and curve counting for arbitrary projective 
    toric surfaces. For $h$-transverse toric surfaces, one can consider a particular choice of tropical incidence
conditions, known as ``vertically stretched", for which the combinatorics of the tropical curves can be encoded by floor diagrams. This is the approach followed in \cite{MR2359091, MR2500574}.
An alternative and more direct way to 
understand the relation between floor diagrams and curve counting 
relies on relative Gromov-Witten theory. 
Indeed, relative Gromov--Witten theory \cite{MR1882667}
allows the definition of counts of curves in $\PP^2$ and Hirzebruch surfaces with tangency conditions along smooth divisors, and then 
floor diagrams naturally appear
\cite{MR3345189, MR3658147}
as describing the combinatorics of successive applications of the degeneration formula in relative Gromov-Witten theory \cite{MR1938113}. 

In this paper, we investigate the connection between counts of complex curves and the $q$-refined counts of floor diagrams introduced by
Block and G\"ottsche \cite{MR3579972}. 
The $q$-refined counts of floor diagrams 
are Laurent  
polynomials in a variable $q$, and they reduce to the ordinary integral counts of floor diagrams for $q=1$.
In \cite{MR3904449}, we established a $q$-refined version of Mikhakin's correspondence theorem, relating $q$-refined counts of tropical curves in 
$\R^2$ \cite{MR3453390} and generating series of higher genus log Gromov-Witten invariants of toric surfaces with insertion of a lambda class after the change of variables $q=e^{iu}$. On the other hand,
for $h$-transverse toric surfaces, the ``vertically stretched" limit connects $q$-refined counts of tropical curves and $q$-refined counts of floor diagrams, as in the unrefined case.
Therefore, \cite{MR3904449} can be used to give an understanding 
based on $q$-refined tropical geometry of the relation between $q$-refined floor diagrams and curve counting.
The goal of this paper is to present an alternative and more direct  understanding of the relation between $q$-refined
floor diagrams and curve counting based on relative Gromov-Witten theory. We show the following result (we refer to Theorem \ref{main_thm_precise} for the precise statement).

\begin{thm} \label{main_thm0}
For $\PP^2$ and Hirzebruch surfaces, $q$-refined counts of floor diagrams are,
after the change of variables $q=e^{iu}$, generating series of higher genus relative Gromov-Witten invariants with insertion of a lambda class.
\end{thm}

Relative Gromov-Witten invariants of a surface 
$S$ with insertion of a lambda class can be naturally viewed as 
equivariant relative Gromov-Witten invariants of the 
$3$-fold $S \times \A^1$ (see \S \ref{section_dim_3}). 
Remarkably, the conjectural correspondence between Gromov-Witten invariants and 
Pandharipande-Thomas stable pair invariants for $3$-folds 
\cite{mnop1, mnop2, pt} is also formulated in terms of a change of variables $q=e^{iu}$. 
As this correspondence is known for the equivariant relative theories of toric $3$-folds \cite{moop, mpt}, we can rephrase Theorem \ref{main_thm0} as follows (see Theorem \ref{thm_main_pt} for the precise statement).

\begin{thm} \label{thm_dt_intro}
For $S=\PP^2$ or a Hirzebruch surface, $q$-refined counts of 
floor diagrams are equivariant relative Pandharipande-Thomas stable pair invariants of the $3$-fold $S \times \A^1$. 
\end{thm}

Tropical computations of higher genus Gromov-Witten invariants and of some stable pair invariants of toric 3-folds were done previously by Brett Parker in the framework of exploded manifolds \cite{parker2017three}. The main result of our paper can be viewed as an example of the tropical/Gromov-Witten correspondence of \cite{parker2017three} for which the tropical side can be explicitly described in terms of floor diagrams, and for which the stable pair reformulation can be easily stated.

\subsection{Structure of the proof of Theorem \ref{main_thm0}}
As in the unrefined case, the combinatorics of the floor diagrams
captures successive applications of the degeneration formula in Gromov-Witten theory. The non-trivial step to 
prove Theorem \ref{main_thm0}
is to evaluate the 
contribution to the curve counts of the various vertices of each floor diagram.
In the unrefined case, these contributions are all trivially equal to $1$. 
In the $q$-refined case, we need to compute explicitly a family of relative Gromov-Witten
invariants with lambda class insertion for Hirzebruch surfaces. The computation of these
invariants in \S \ref{section_computation} is the main new technical content of this paper and is done by an induction whose each step
requires the application of the degeneration formula for relative Gromov-Witten invariants and the
explicit knowledge of relative Gromov-Witten invariants of $\PP^1$. Perhaps curiously, 
the cancellation of terms necessary for the induction step is the power series version of the identity
\begin{equation}\exp (\log(1+x)) = 1+x\,.
\end{equation}

\subsection{Refined Fock spaces} Cooper and Pandharipande 
\cite{MR3653237} remarked that the combinatorics of the degeneration formula
in relative Gromov-Witten theory for $\PP^1 
\times \PP^1$ and $\PP^2$ can be nicely encoded into an operator formalism in Fock space. This approach has been recently generalized to Hirzebruch surfaces by 
Cooper \cite{cooper2017fock}.
Block and G\"ottsche \cite{MR3579972} generalized this remark to $h$-transverse toric surfaces
by recognizing that the floor diagrams were the Feynman diagrams of the operator formalism 
in Fock space. They also remarked that the $q$-refined floor diagrams can still be interpreted 
as Feynman diagrams of a $q$-deformed operator formalism in Fock space. It follows that 
Theorem \ref{main_thm0} can be equivalently phrased in terms of the $q$-deformed operator formalism in 
Fock space: this operator formalism computes, after the change of variables $q=e^{iu}$, generating series of higher genus relative Gromov-Witten invariants with insertion of a lambda class. We refer to Corollary 3.7 of \cite{MR3579972}
for the explicit formulas in terms of
$q$-deformed operator formalism in Fock space for $\PP^2$ and Hirzebruch surfaces.

\subsection{Log invariants} This paper is logically independent of \cite{MR3904449}: it is phrased entirely into the framework of relative Gromov-Witten theory
along smooth divisors \cite{MR1882667}, 
and does not require the log technology used in \cite{MR3904449}.
In particular, we hope that the present paper could be viewed as a more accessible introduction
to the set of ideas presented in \cite{MR3904449}. 

The combination of Theorem \ref{main_thm0} with the main result of \cite{MR3904449}
produces an interesting result. As both relative and log Gromov-Witten invariants of $\PP^2$ and Hirzebruch surfaces are computed by the same $q$-refined floor diagrams, we obtain
a non-trivial relation between them. We give a precise statement in 
Theorem \ref{thm_log}. This relation could probably be obtained directly using a
degeneration argument in a log context, but it is interesting that tropical geometry gives an alternative argument. 
In the unrefined case, similar remarks are made
in \cite{cavalieri2017counting} and \cite{cooper2017fock}.

\subsection{The $q$-refined Abramovich-Bertram formula} A classical formula, due to Abramovich-Bertram \cite{MR1837110} 
in genus zero and to Vakil in higher genus \cite{MR1771228}, relates the enumerative
geometries of the Hirzebruch surfaces
$\F_0$ and $\F_2$.
Motivated by the fact that the same formula holds for Welschinger counts of real curves, Brugall\'e
\cite{brugalle2018invariance} has recently conjectured that the same formula holds at the level of the corresponding $q$-refined Block-Göttsche tropical invariants.
We give a proof of this conjecture in \S 
\ref{section_brugalle_conj}
(see Corollary \ref{cor_AB} for the precise statement). 

\begin{thm} \label{thm_f0_f2_intro}
The $q$-refined counts of floor diagrams for 
$\F_0$ and $\F_2$ are related by a $q$-refinement of the Abramovich-Bertram formula.
\end{thm}

Whereas the statement of Theorem \ref{thm_f0_f2_intro} is an identity between $q$-refined combinatorial counts,
and so possibly accessible by a purely combinatorial argument, our proof is geometric: using Theorem \ref{main_thm0}, we rephrase 
Theorem \ref{thm_f0_f2_intro} as a relation between relative Gromov-Witten invariants, and the result then follows from the degeneration formula in Gromov-Witten theory.

\subsection{Plan of the paper}
In \S
\ref{section_floor_diagrams}, we review the definitions of 
$h$-transverse toric surfaces, floor diagrams, and $q$-refined counts.
In \S \ref{section_relative_gw_theory}, we fix our notations for relative Gromov-Witten invariants with insertion of a lambda class.
The technical heart of the paper is \S \ref{section_computation}, in which we evaluate explicitly a family of relative Gromov-Witten invariants of Hirzebruch surfaces with insertion of a lambda class. 
In \S \ref{section_main_result}, we combine the calculations of 
\S \ref{section_computation} with a degeneration argument to establish our main result, Theorem \ref{main_thm_precise}, relating $q$-refined counts of floor diagrams and higher genus Gromov-Witten invariants.
In \S \ref{section_dim_3}, we explain how to rephrase this result 
from a $3$-dimensional point of view in terms of stable pair invariants.
In \S \ref{section_log}, we combine Theorem \ref{main_thm_precise}
with the main result of \cite{MR3904449} to get Theorem \ref{thm_log}, a comparison result
between relative and log Gromov-Witten invariants. Finally, in \S
\ref{section_brugalle_conj}, we give,
as application of Theorem
\ref{main_thm_precise}, the proof of Theorem \ref{thm_f0_f2_intro}, that is, that
Block-Göttsche invariants of $\F_0$ and $\F_2$ are related by the Abramovich-Bertram formula, as conjectured in \cite{brugalle2018invariance}.

\subsection{Acknowledgements.}

The question to give an analogue of \cite{MR3904449} in the context of 
floor diagrams was asked
by Lothar G\"ottsche during a discussion about \cite{MR3904449} in Trieste in June 2017. I obtained
the key part of the present paper (the proof by induction of Theorem \ref{key_thm}) in the
days following this discussion. 
I also thank Rahul Pandharipande for several useful discussions on related topics and Hülya Argüz for her help with the figures. Finally, I thank the anonymous referee for many useful comments and suggestions that have 
greatly contributed to the improvement of the exposition. During the preparation of this paper I was supported by Dr.\ Max R\"ossler, the Walter Haefner Foundation and the ETH Z\"urich Foundation.

\section{Floor diagrams} \label{section_floor_diagrams}
We review
$h$-transverse toric surfaces in \S \ref{section_h_transverse}, floor diagrams 
in \S \ref{section_floor}, and $q$-refined counts of floor diagrams in \S \ref{section_q_refined}.

\subsection{$h$-transverse toric surfaces} \label{section_h_transverse}

\begin{figure}
\center{\scalebox{.4}{\input{Db.pspdftex}}}
\caption{On the left, the $h$-transverse balanced collection 
$\Delta_d^{\PP^2}$. On the right, the corresponding toric surface $\PP^2$.}
\label{Fig:Db}
\end{figure}

\begin{defn} \label{def_h_transverse}
Let $\Delta$ be a balanced collection of vectors in $\Z^2$,
that is, a finite collection of vectors in 
$\Z^2-\{0\}$ summing up to zero. 
Following  \cite{MR2500574},  $\Delta$ is 
\emph{$h$-transverse} if for every $v =(v_x,v_y)\in \Delta$,
we have either
 $v_x =\pm 1$, or both $v_x=0$ and $v_y=\pm 1$.
    
In other words, $\Delta$ is $h$-transverse if all non-vertical vectors in $\Delta$ have an horizontal component equal to $+1$ or $-1$
and all the vertical vectors in $\Delta$ are of the form $(0,1)$
or $(0,-1)$, .
\end{defn}

The property of being 
$h$-transverse is not invariant under the natural action 
of $GL(2,\Z)$ on $\Z^2$: it depends on the notion of horizontal and vertical directions in $\Z^2$.

\begin{defn} \label{def_X_delta}
Given $\Delta$ a $h$-transverse balanced collection of vectors in $\Z^2$ as in Definition 
\ref{def_h_transverse}, we denote by $X_{\Delta}$ 
the toric surface 
over $\C$ whose fan has for set of rays
\begin{equation}
 \{ \R_{\geq 0} v\,|\, v \in \Delta\}\,.   
\end{equation}
\end{defn}

If $(0,-1)$ appears in $\Delta$,
we denote by $D_b$
the toric divisor of $X_{\Delta}$ dual to the ray $\R (0,-1)$,
else we set $D_b \coloneqq \emptyset$.
If $(0,1)$ appears in $\Delta$, we denote by $D_t$ the toric divisor of $X_{\Delta}$
dual to the ray $\R_{\geq 0} (0,1)$, else we set 
$D_t \coloneqq \emptyset$. The indices ``$b$" and ``$t$" in $D_b$ and $D_t$
refer respectively to ``bottom" and ``top".
We denote by $d_b$ (resp.\ $d_t$) the number of occurrences of $(0,-1)$ (resp. $(0,1)$) in $\Delta$.

Let 
\begin{equation} \label{eq_delta_l}
\Delta_l \coloneqq \{ v=(v_x,v_y) \in \Delta\,|\,v_x=-1\} \,,
\end{equation}
and 
\begin{equation} \label{eq_delta_r}
\Delta_r \coloneqq \{v=(v_x,v_y) \in \Delta\,|\,v_x=1\} \,,
\end{equation}
where the indices ``$l$" and ``$r$" refer respectively to ``left"
and ``right".

By Definition \ref{def_h_transverse}, 
$\Delta_l \cup \Delta_r$ is the subset of non-horizontal vectors in 
$\Delta$. As $\Delta$ is balanced, 
 $\Delta_l$ and $\Delta_r$ have the same cardinality, 
which we denote $h$ and call the ``height" of $\Delta$.
It follows from Definition \ref{def_h_transverse} that
\begin{equation} \label{eq_id}
d_b + d_t + 2h = |\Delta|\,,
\end{equation}
where $|\Delta|$ is the cardinality of $\Delta$.

\begin{defn} \label{def_beta_delta}
By standard toric geometry
(see \cite[\S 3.4]{Fultontoric}), there exists a unique homology class $\beta_{\Delta} \in H_2(X_{\Delta},\Z)$ such that
for every ray $\R_{\geq 0}v$ of the fan of $X_\Delta$, with primitive 
$v\in \Z^2$, the intersection number $\beta_{\Delta} \cdot D_v$ of 
$\beta_\Delta$ with the divisor $D_v$ of $X_\Delta$ dual to $\R_{\geq 0}v$
is equal to the number of occurrences of $v$ in $\Delta$. In particular, we have $\beta_{\Delta} \cdot D_b =d_b$
and $\beta_{\Delta} \cdot D_t = d_t$.
\end{defn}

In this paper, we focus on the following two examples.
\begin{ex} \label{ex_P2}
    Let $d$ be a positive integer and let $\Delta^{\PP^2}_d$ be the balanced collection of vectors in $\Z^2$ consisting of $d$ copies of $(-1,0)$, $d$ copies of $(0,-1)$
    and $d$ copies of $(1,1)$, see Figure \ref{Fig:Db}. Then $X_{\Delta^{\PP^2}_d} = \PP^2$, $d_b=d$, $d_t=0$, $h=d$
    and $\beta_{\Delta^{\PP^2}_d}$ is the class of a degree $d$ curve in $\PP^2$.
\end{ex}

\begin{figure}
\center{\scalebox{.4}{\input{Fk.pspdftex}}}
\caption{On the left, the $h$-transverse balanced collection 
$\Delta_{h,d}^{\F_k}$. On the right, the corresponding toric surface $\F_k$.}
\label{Fig:Fk}
\end{figure}

\begin{ex} \label{ex_hirzebruch}     
Let $k$ be an integer and let $d$ and $h$ be non-negative integers. Let $\Delta^{\F_k}_{h,d}$ be the balanced collection of vectors in $\Z^2$ consisting of $d+kh$ copies of $(0,-1)$, $d$ copies of $(0,1)$, $h$ copies of $(-1,0)$ and $h$ copies of $(1,k)$, see Figure \ref{Fig:Fk}. Then $X_{\Delta^{\F_k}_{h,d}}$ is the Hirzebruch surface $\F_k=\PP(\cO_{\PP^1} \oplus 
    \cO_{\PP^1}(k))$, $d_b=d+kh$, $d_t=d$, $D_b$ is the toric divisor $D_k$
    such that $D_k^2=k$ and $D_t$ is the toric divisor $D_{-k}$ such that $D_{-k}^2=-k$.
    Denoting by $F$ the class of a $\PP^1$-fiber of the natural projection $p \colon \F_k \rightarrow \PP^1$, we have 
    \begin{equation}\label{eq_beta_hirzebruch}
    \beta_{\Delta^{\F_k}_{h,d}} \cdot D_{k}=d+kh\,,\,\,\, \beta_{\Delta^{\F_k}_{h,d}} \cdot D_{-k}=d\,,\,\,\,
    \beta_{\Delta^{\F_k}_{h,d}} \cdot F=h\,, 
    \end{equation}
    and so
    \begin{equation}\beta_{\Delta^{\F_k}_{h,d}} =h D_{k}+d F \,.
    \end{equation}
\end{ex}

\subsection{Floor diagrams} \label{section_floor}

In this paper, we adopt the following conventions on graphs.
\emph{Graphs} are connected, have finitely many vertices, 
finitely many \emph{bounded} edges, and finitely many \emph{unbounded} edges.
A bounded edge connects two distinct vertices, whereas an unbounded edge is incident to a single vertex. A \emph{weighted} graph is a graph endowed with a choice a positive integer $w_E$ for every edge $E$, called the \emph{weight} of $E$.
An \emph{oriented} graph is a graph endowed with a choice of orientation of every edge.

Up to notational details, the following definitions are due to  Brugall\'e and Mikhalkin \cite{MR2500574}. 

\begin{defn} \label{def_floor_diagram}
Let $\Delta$ be a $h$-transverse balanced collection of vectors in $\Z^2$
as in Definition \ref{def_h_transverse} and $n$ a nonnegative integer.
A \emph{$(\Delta,n)$-floor diagram} $\fD$ is the data of a weighted oriented graph $\Gamma$ and of bijections
\begin{align} \label{eq_bij}
    v_l \colon & V(\Gamma) \simeq \Delta_l \\
    &V \mapsto v_l(V)\,, \nonumber
\end{align}
and 
\begin{align}
    v_r \colon & V(\Gamma) \simeq \Delta_r \\
    &V \mapsto v_r(V)\,, \nonumber
\end{align}
where $V(\Gamma)$ is the set of vertices of $\Gamma$ and 
$\Delta_l$, $\Delta_r$ are as in \eqref{eq_delta_l}-\eqref{eq_delta_r},
such that
\begin{itemize}
    \item[(i)] the oriented graph $\Gamma$ is acyclic, that is, does not contain any oriented cycle,
    \item[(ii)] the first Betti number of $\Gamma$ equals $g_{\Delta,n} \coloneqq
    n+1-|\Delta|$, where $|\Delta|$ is the cardinality of $\Delta$,
    \item[(iii)] there are exactly $d_b$ incoming unbounded edges and 
    $d_t$ outgoing unbounded edges, and all of them have weight $1$,
    \item[(iv)] for every $V$ vertex of $\Gamma$, writing \begin{equation} v_l(V)=(-1,v_l(V)_y) \in \Z^2\end{equation}
    and \begin{equation} v_r(V)=(1,v_r(V)_y) \in \Z^2\,,
    \end{equation} the sum of weights of incoming edges minus the sum of weights of outgoing edges is equal to $v_l(V)_y + v_r(V)_y$. 
\end{itemize}
\end{defn}

\begin{lem}
Let $\fD$ be a 
$(\Delta, n)$-floor diagram of underlying graph $\Gamma$.
Then 
\begin{equation}
    |V(\Gamma)|+|E(\Gamma)|=n \,,
\end{equation}
where $|V(\Gamma)|$ 
(resp.\ $|E(\Gamma)|$) is the cardinality of the set of vertices of $\Gamma$.
\end{lem}

\begin{proof}
Let $|E^b(\Gamma)|$ (resp.\ $|E^\infty(\Gamma)|$) be the cardinality of the set of bounded (resp.\ unbounded) edges of $\Gamma$, and $h$ the height of 
$\Delta$. We have $|V(\Gamma)|=h$ by \eqref{eq_bij}, $|E^\infty(\Gamma)|=d_b+d_t$ by Definition 
\ref{def_floor_diagram}(iii), and 
$|E^b(\Gamma)|-|V(\Gamma)|+1=g_{\Delta,n}=n+1-|\Delta|$ by 
Definition \ref{def_floor_diagram}(i)-(ii).
So $|E^b(\Gamma)|=n-|\Delta|+h$ and
\begin{equation}
    |V(\Gamma)|+|E(\Gamma)|=
    |V(\Gamma)|+|E^b(\Gamma)|+|E^\infty(\Gamma)|
    =h+(n-|\Delta|+h)+d_b+d_t\,,
\end{equation}
and the result follows from \eqref{eq_id}.
\end{proof}

\begin{defn} \label{def_marking}
Let $\fD$ be a $(\Delta, n)$-floor diagram, of underlying graph $\Gamma$, with set of vertex $V(\Gamma)$ and set of edges $E(\Gamma)$.
As $\Gamma$ is acyclic, oriented edges of $\Gamma$ define a partial ordering on $V(\Gamma) \cup E(\Gamma)$.
A \emph{marking} of $\fD$ is an increasing bijection between the ordered set $\{1,\dots,n\}$ and the partially ordered set 
$V(\Gamma) \cup E(\Gamma)$.
\end{defn}

\begin{defn}
Two marked floor diagrams are \emph{isomorphic} if there exists an homeomorphism of their underlying graphs, compatible with the orientations, the weights, the bijections 
$v_l$ and $v_r$, and the markings.
\end{defn}

\begin{defn} \label{def_mult}
The \emph{multiplicity} of a marked floor diagram $\fD$ is
the positive integer 
\begin{equation} m_\fD \coloneqq \prod_E w_E^2 \,,
\end{equation}
where the product is over the edges of $\fD$ and $w_E$ is the weight of 
the edge $E$.
\end{defn}
The multiplicity of a marked floor diagram only depends on its isomorphism class.

\begin{defn} \label{def_N_floor}
The \emph{count with multiplicity} of marked 
$(\Delta,n)$-floor diagrams is 
\begin{equation} 
N^{\Delta,n}_{\floor} \coloneqq \sum_{\fD} m_\fD  \,,
\end{equation}
where the sum is over the isomorphism classes of marked 
$(\Delta,n)$-floor diagrams.
\end{defn}

The main result of Brugall\'e and Mikhalkin \cite{MR2500574} is that the count $N^{\Delta,n}_{\floor}$ with multiplicity 
of marked 
$(\Delta,n)$-floor diagrams
coincides with the number of curves of genus $g_{\Delta,n}$ and class 
$\beta_{\Delta}$ (see Definition \ref{def_beta_delta}) in the toric surface $X_\Delta$ (see Definition \ref{def_X_delta}), passing through $n$ fixed points in general position and intersecting transversally the toric divisors $D_b$ and 
$D_t$.

\subsection{$q$-refined counts of floor diagrams}
\label{section_q_refined}
For every nonnegative integer $m$,
we define the $q$-integer $[m]_q$ by
\begin{equation}\label{eq_q_integer}
[m]_q \coloneqq \frac{
q^{\frac{m}{2}} - q^{-\frac{m}{2}}
}{q^{\frac{1}{2}} - q^{-\frac{1}{2}}}
=q^{-\frac{m-1}{2}} (1+q+\dots+q^{m-1}) \,.
\end{equation}
It is a Laurent polynomial in a formal 
variable $q^{\frac{1}{2}}$, reducing to the integer $m$ in the limit $q^{\frac{1}{2}}
\rightarrow 1$.
The following definitions are due to Block and Göttsche \cite{MR3579972}.

\begin{defn} \label{def_q_mult}
The \emph{$q$-refined multiplicity} of a marked floor diagram $\fD$ is 
\begin{equation}m_\fD(q^{\frac{1}{2}}) \coloneqq \prod_E [w_E]_q^2 \,,\end{equation}
where the product is over the edges of $\fD$ and $w_E$ is the weight of 
the edge $E$.
\end{defn}
The $q$-refined multiplicity of a marked floor diagram only depends on its isomorphism class.  

\begin{defn} \label{def_N_delta_q}
The \emph{count with $q$-multiplicity} of 
$(\Delta, n)$-floor diagrams is 
\begin{equation} \label{q_mult}
N^{\Delta,n}_{\floor}(q^{\frac{1}{2}}) = \sum_{\fD} m_\fD(q^{\frac{1}{2}})\,,
\end{equation}
where the sum is over the isomorphism classes of marked $(\Delta,n)$-floor diagrams.
\end{defn}

In the unrefined limit $q^{\frac{1}{2}}
\rightarrow 1$, the $q$-refined multiplicity $m_\fD(q^{\frac{1}{2}})$
in Definition \ref{def_q_mult}
reduces to the ordinary multiplicity 
$m_\fD$ in Definition \ref{def_mult}, 
and so the $q$-refined count 
$N^{\Delta,n}_{\floor}(q^{\frac{1}{2}})$ in Definition
\ref{def_N_delta_q}
reduces to the unrefined count
$N^{\Delta,n}_{\floor}$ in Definition \ref{def_N_floor}.

\section{Relative Gromov-Witten theory}
\label{section_relative_gw_theory}

In \S \ref{section_relative_gw_general}, we review the general framework of relative Gromov-Witten theory \cite{MR1882667}\cite[\S 3.4]{mnop2}
and 
we introduce notations for relative Gromov-Witten invariants of geometries relative to two disjoint smooth divisors. 
We describe lambda classes in \S \ref{section_lambda_classes} 
and we prove a gluing formula for lambda classes in 
\S \ref{section_top_lambda_gluing}.
In \S \ref{section_vanishing}, we prove the vanishing of a class of 
relative Gromov-Witten invariants of surfaces with insertion of a lambda class.

\subsection{Relative Gromov-Witten invariants}
\label{section_relative_gw_general}

Let $X$ be a smooth projective variety over $\C$ and 
$D_1, D_2 \subset X$ two disjoint smooth divisors.
Let $g, r \in \Z_{\geq 0}$ and  $\beta \in H_2(X,\Z)$ a curve class such that 
$\beta \cdot D_1 \geq 0$ and $\beta \cdot D_2 \geq 0$.
Let
$\eta^1=(\eta_j^1)_{1\leq j \leq \ell(\eta^1)}$, 
$\eta^2=(\eta_j^2)_{1\leq j \leq \ell(\eta^2)}$ be two (unordered) partitions of $\beta \cdot D_1$ and $\beta \cdot D_2$ respectively.
We choose an ordering of the parts of $\eta^1$ and $\eta^2$ and we denote by
$\vec{\eta}^1=(\vec{\eta}_j^1)_{1\leq j \leq \ell(\eta^1)}$ and 
$\vec{\eta}^2=(\vec{\eta}_j^2)_{1\leq j \leq \ell(\eta^2)}$
the resulting ordered partitions.

The moduli stack of relative stable maps
\begin{equation} \label{eq_moduli_space}
\overline{M}_{g,r}(X/D_1\cup D_2,\beta,\vec{\eta}^1,\vec{\eta}^2)
\end{equation}
is a proper Deligne-Mumford stack which compactifies the moduli space of genus $g$ class $\beta$
stable maps 
\begin{equation}f \colon  (C,x_1,\dots,x_r,y_1,\dots,y_{\ell(\vec{\eta}^1)},
z_1,\dots,z_{\ell(\vec{\eta}^2)}) \rightarrow X 
\end{equation}
such that $f(x_j) \notin D$, $f(y_j) \in D$, $f(z_j) \in D$
and the contact order of $f$ along $D_1$ (resp.\ $D_2$) at the marked point $y_j$ (resp.\ $z_j$) is $\vec{\eta}^1_j$ (resp.\ $\vec{\eta}^2_j$) \cite{MR1882667}. In general, the target of a relative stable map is allowed to be an expanded degeneration
of $X$ along $D_1$ and $D_2$  \cite{MR1882667}.
The virtual dimension of the moduli stack of relative stable maps is \begin{align} \label{eq_vdim}
   & \vdim \overline{M}_{g,r}(X/D_1\cup D_2,\beta,\vec{\eta}^1,\vec{\eta}^2) \\
    &= (1-g)(\dim X-3)+\int_\beta c_1(T_X)+r -\sum_{j=1}^{\ell(\eta^1)} (\eta_j^1-1) -\sum_{j=1}^{\ell(\eta^2)} (\eta_j^2-1) \nonumber \,,
\end{align}
where $c_1(T_X) \in H^2(X,\Z)$ is the first Chern class of the tangent bundle $T_X$ of $X$.

Relative Gromov-Witten invariants of $X/D_1 \cup D_2$ are defined by integration against the virtual fundamental class of the moduli stack of relative stable maps.
Given even degree cohomology classes
\begin{equation}
 \alpha_1,\dots,\alpha_r \in H^{*}(X,\Q) \,,
\end{equation}
\begin{equation}
 \delta^1=(\delta_j^1)_{1\leq j \leq \ell(\eta^1)} \,\,\,\text{with}\,\,\,\delta_j^1 \in H^{*}(D_1,\Q) \,,
\end{equation}
\begin{equation}
 \delta^2=(\delta_j^2)_{1\leq j \leq \ell(\eta^2)} \,\,\,\text{with}\,\,\,\delta_j^2 \in H^{*}(D_2,\Q) \,,
\end{equation}
\begin{equation}
\gamma \in H^{*}(\overline{M}_{g,r}(X/D_1\cup D_2,\beta,\vec{\eta}^1,\vec{\eta}^2),\Q)\,,
\end{equation}
let 
\begin{align} \label{eq_gw}
&    \langle \eta^2,\delta^2|\, \gamma; \alpha_1,\dots,\alpha_r|\,\eta^1,\delta^1 \rangle_{g,\beta}^{X/D_1 \cup D_2}   \coloneqq \frac{1}{|\Aut(\eta^1,\delta^1)|} \frac{1}{|\Aut(\eta^2,\delta^2)|}\\
&    \times \int_{[\overline{M}_{g,r}(X/D_1\cup D_2,\beta,\vec{\eta}^1,\vec{\eta}^2)]^{\virt}} \gamma \prod_{l=1}^r \ev_r^{*}(\alpha_l) \prod_{m=1}^{\ell(\eta^1)} (\ev_m^1)^{*}(\delta_m^1) 
    \prod_{n=1}^{\ell(\eta^2)} (\ev_n^2)^{*}(\delta_n^2) \nonumber \,.
\end{align}
Here, $|\Aut(\eta^1,\delta^1)|$
(resp.\ $|\Aut(\eta^2,\delta^2)|$) is the order of the group of 
permutation symmetries of the set of pairs $(\eta^1_j,\delta_j^1)$ for
$1 \leq j \leq \ell(\eta^1)$
(resp.\ $(\eta_j^2,\delta_j^2)$ for $1 \leq j \leq \ell(\eta^2)$).
Moreover, $\ev_r$
is the morphism from the moduli stack to $X$ defined by the evaluation at the interior marking $x_r$, 
and $\ev_m^1$ (resp.\ $\ev_n^2$) are the morphisms from the moduli stack to $D_1$ (resp.\ $D_2$) defined by the evaluation at the relative marking 
$y_m$ (resp.\ $z_n$). As we are assuming that all the cohomology classes are even and so commute for the cup-product, the relative Gromov-Witten invariant
\eqref{eq_gw} only depends on the unordered partitions 
$\eta^1$, $\eta^2$, and not on the orderings chosen to define 
$\vec{\eta}^1$, $\vec{\eta}^2$. The automorphism prefactors in 
\eqref{eq_gw} effectively allow us to forget the ordering on the sets of relative markings with given contact order and cohomology classs insertion.

When working with generating series summing over the genus, we will always weight invariants as in \eqref{eq_gw} by $u^{2g-2+\ell(\eta^1)+\ell(\eta^2)}$, where $u$ is a formal variable keeping track of the genus.  
This convention simplifies the writing of the degeneration formula \cite{MR1938113}.

\subsection{Lambda classes}
\label{section_lambda_classes}
In the definition \eqref{eq_gw} of relative Gromov-Witten invariants, 
we allow for the insertion of a class 
$\gamma$ in the cohomology of the moduli stack of relative stable maps. 
In this section, we review a particular family of such cohomology classes
called lambda classes. Relative Gromov-Witten invariants with insertion of a lambda class are the main objects of study of the present paper.

Let $\cM$ be a finite type Deligne-Mumford stack over $\C$.
Let $\pi \colon \cC \rightarrow \cM$ be a family of genus $g$ prestable curves, that is, 
$\pi$ is a flat proper morphism and geometric fibers of $\pi$ 
are connected nodal curves of arithmetic genus $g$. 
The relative dualizing sheaf $\omega_\pi$ of 
$\pi$ is a line bundle on $\cC$ and its pushforward 
$\E_\pi \coloneqq \pi_{*} \omega_\pi$ is a rank $g$ vector bundle on 
$\cM$, called the Hodge bundle.
Following Mumford
\cite[\S 4]{MR717614}, Chern classes of the Hodge bundle are denoted
\begin{equation} \label{eq_lambda}
\lambda_{j,\pi} \coloneqq c_j(\E_\pi) 
\in H^{2j}(\cM,\Q)\,,\,\,\,\,\, 0\leq j \leq g \,,\end{equation}
and called \emph{lambda classes}. The 
\emph{top lambda class} is $\lambda_{g,\pi} \in H^{2g}(\cM,\Q)$.

According to \cite[(5.4)-(5.5)]{MR717614}, the total Chern classes 
$c(\E_\pi)=\sum_{j=0}^g \lambda_{j,\pi}$ and $c(\E^\vee_\pi)=\sum_{j=0}^g (-1)^j \lambda_{j,\pi}$
obey the identity $c(\E_\pi)\,c(\E^\vee_\pi)=1$. In particular, taking the component of complex degree $2g$, we have 
\begin{equation}\label{eq_lambda_square_zero}
    \lambda_{g,\pi}^2=
\begin{cases}    
    1 &\mathrm{if}~g=0\\
    0 &\mathrm{if}~g>0\,.
\end{cases}
\end{equation}

In the following sections of this paper, we apply the construction of 
lambda classes for $\cM$ a moduli stack of relative stable maps and 
$\pi \colon \cC \rightarrow \cM$ the universal source curve.
In such case, we simply denote $\lambda_j$ for $\lambda_{j,\pi}$.

\subsection{Lambda classes and gluing}
\label{section_top_lambda_gluing}
We review the behavior of lambda classes under gluing, following the exposition given in \cite{MR3904449}.

\begin{prop} \label{lem_gluing1}
Let $\cM$ be a finite type Deligne-Mumford stack over 
$\C$. 
Let $\Gamma$ be a graph of first Betti number 
$g_\Gamma$. For every vertex $V$ of $\Gamma$, let
$\pi_V \colon \cC_V
\rightarrow \cM$ be a family of genus $g_V$ prestable curves.
For every edge $E$ of $\Gamma$, connecting vertices $V_1$ and $V_2$, let $s_{E,1}$ and $s_{E,2}$ be two sections of $\pi_{V_1}$ and 
$\pi_{V_2}$ avoiding nodes.
Denote by $\pi \colon \cC \rightarrow \cM$ the family of genus $g \coloneqq g_\Gamma +\sum_V g_V$
prestable curves obtained by gluing together transversally 
the sections $s_{V_1,E}$ and $s_{V_2,E}$ for every edge $E$ of $\Gamma$.
Then, for every $0\leq j\leq g$,
\begin{equation} \label{eq_lambda_gluing}
\lambda_{j,\pi} = \sum_{\substack{(j_V)_V\\
0 \leq j_V \leq g_V\\ \sum_V j_V=j}} 
\prod_V \lambda_{j_V,\pi_V}\,.
\end{equation}
In particular, $
\lambda_{g-j,\pi}=0$
if $g_\Gamma>j$.
\end{prop}

\begin{proof}
We denote by $V(\Gamma)$ (resp.\ $E(\Gamma)$) the set of vertices
(resp.\ edges) of $\Gamma$.
For every edge $E \in E(\Gamma)$, let $s_E \colon \cM \rightarrow \cC$
be the family of nodes defined by $E$.  
Applying $R \pi_*$ to the short exact sequence 
\begin{equation} 0 \rightarrow \cO_\cC
\rightarrow \bigoplus_{V 
\in V(\Gamma)} \cO_{\cC_V}
\rightarrow \bigoplus_{E \in E(\Gamma)} \cO_{s_E(\cM)}
\rightarrow 0\,,\end{equation}
we obtain the long exact sequence
\begin{equation}\label{eq_long_exact} 0 \rightarrow \pi_* \cO_{\cC}
\rightarrow \bigoplus_{V \in V(\Gamma)} 
\pi_* \cO_{\cC_V}
\rightarrow 
\bigoplus_{E \in E(\Gamma)}
\pi_* \cO_{s_E(\cM)}
 \rightarrow 
R^1 \pi_* \cO_\cC 
\xrightarrow{f} 
\bigoplus_{V \in V(\Gamma)} R^1 \pi_*
\cO_{\cC_V} \rightarrow 0 \,.\end{equation}
The kernel of the map $f$ in \eqref{eq_long_exact}
is a free sheaf of rank
$|E(\Gamma)|-|V(\Gamma)|+1=g_\Gamma$.
Using Serre duality, we find the short exact sequence
\begin{equation} \label{eq_short} 0 \rightarrow 
\bigoplus_{V \in V(\Gamma)}
\E_{\pi_V}
\rightarrow \E_\pi
\rightarrow \cO^{\oplus g_\Gamma}
\rightarrow 0\,.\end{equation}
The result follows from the Whitney sum formula for Chern classes applied to \eqref{eq_short} and the vanishing of the Chern classes of the trivial vector bundle $\cO^{\oplus g_\Gamma}$.
\end{proof}

\subsection{A vanishing result for surfaces}
\label{section_vanishing}
In this section, we prove as a consequence of \eqref{eq_lambda_square_zero} the vanishing of a particular class of relative Gromov-Witten invariants of surfaces.
We use the notations introduced in 
\S \ref{section_relative_gw_general}.

\begin{lem} \label{lem_vanishing}
Let $X$ be a smooth projective surface such that 
$H^1(X,\cO_X)=0$ and $D_1$, $D_2$ two disjoint smooth divisors of $X$. Let $\beta \in H_2(X,\Z)$ be an effective curve class such that $\beta \cdot \beta=0$ and the curves of class $\beta$ form 
a 1-dimensional linear system of smooth rational curves in $X$. Let $k \in \Z_{\geq 1}$. Assume that for every map
$f \colon Y \rightarrow X$ such that $Y$ is a connected curve
and $f_{*}[Y]=k \beta$, there exists a curve $C\subset X$ of class $\beta$ and a map $g \colon Y \rightarrow C$ such that 
$f$ is the composition of $g$ with the inclusion 
$C \subset X$.
Then, for every $g >0$
and partition
$\eta^1$ (resp.\ $\eta^2$) of $k\beta \cdot D_1$
(resp.\ $k\beta \cdot D_2$), we have 
\begin{equation} \langle \eta^2,\delta^2|\, (-1)^g \lambda_g; \alpha_1|\,\eta^1,\delta^1 \rangle_{g,k\beta}^{X/D_1 \cup D_2}  =0\,, 
\end{equation}
where
\begin{itemize}
\item[(i)]$\delta^1=(\delta_j^1)_{1\leq j\leq \ell(\eta^1)}$
is given by $\delta_j^1=1 \in H^0(D^1,\Z)$ for all $j$,  \item[(ii)]$\delta^2=(\delta_j^2)_{1\leq j\leq \ell(\eta^2)}$
is given by $\delta_j^2=1 \in H^0(D^2,\Z)$ for all $j$,
\item[(iii)]$\alpha_1 \in H^4(X,\Z)$ is the cohomology class Poincaré dual to a point,
\item[(iv)] the class $\gamma$ in \eqref{eq_gw} on the moduli space of relative stable maps is taken to be $(-1)^g \lambda_g$, where $\lambda_g$ is the top lambda class as in 
\eqref{eq_lambda}.
\end{itemize}
\end{lem}

\begin{proof}
As the linear system of curves of class $\beta$ is of dimension $1$, there exists $p \in X$ such that $p \notin D_1 \cup D_2$ and $p$ is not a base point of the linear system. 
For such point $p$, there exists a unique smooth rational curve $C\simeq \PP^1 \subset X$ of class $\beta$ passing through $p$. The composition of a relative stable map to $C$ with the inclusion $C \subset X$ defines a closed embedding
\begin{equation} \label{eq_substack}
\overline{M}_{g,1}(C/((D_1\cup D_2)\cap C),k[C],\vec{\eta}^1,\vec{\eta}^2) 
\subset \overline{M}_{g,1}(X/D_1\cup D_2,k\beta,\vec{\eta}^1,\vec{\eta}^2)\,.
\end{equation}
By our assumption, \eqref{eq_substack} is exactly the substack of relative stable maps whose image contains $p$. 

On the other hand, the perfect obstruction theories for relative stable maps to $C$ and to $X$ differ by the top Chern class of the bundle whose fiber over the relative stable map 
$f \colon Y \rightarrow X$ is $H^1(Y,f^{*}N_{C|X})$, where 
$N_{C|X}$ is the normal bundle to $C$ in $X$. As $C \simeq \PP^1$ and $\beta \cdot \beta=0$, we have $N_{C|X}=\cO_C$, and so $H^1(Y, f^{*}N_{C|X})=H^0(C,\omega_{C})^\vee$ by Serre duality. Thus, the perfect obstruction theories differ by 
$c_g(\E)=(-1)^g \lambda_g$. Therefore,
\begin{align}
    &\alpha_1 \cap [ \overline{M}_{g,1}(X/D_1\cup D_2,\beta,\vec{\eta}^1,\vec{\eta}^2)]^\virt \\
    &=(-1)^g \lambda_g \cap \,\alpha_1' \,\cap [\overline{M}_{g,1}(C/((D_1\cup D_2)\cap C),k[C],\vec{\eta}^1,\vec{\eta}^2) ]^{\virt} \nonumber
\end{align}
where $\alpha_1' \in H^2(C,\Z)$ is the cohomology class Poincaré dual of a point. Hence
\begin{equation}
    \langle \eta^2,\delta^2|\, (-1)^g \lambda_g; \alpha_1|\,\eta^1,\delta^1 \rangle_{g,k\beta}^{X/D_1 \cup D_2} = 
    \langle \eta^2,\delta^2|\, \lambda_g^2; \alpha_1'|\,\eta^1,\delta^1 \rangle_{g,k[C]}^{C/((D_1 \cup D_2) \cap C)}\,,
\end{equation}
and the vanishing follows from the fact \eqref{eq_lambda_square_zero} that 
$\lambda_g^2=0$ for $g>0$.
\end{proof}

\begin{lem} \label{lem_vanishing_2}
Let $X$ be a smooth projective surface such that 
$H^1(X,\cO_X)=0$ and $D_1$, $D_2$ two disjoint smooth divisors of $X$. Let $\beta \in H_2(X,\Z)$ be an effective curve class such that $\beta \cdot \beta=0$ and the curves of class $\beta$ form 
a 1-dimensional linear system of smooth rational curves in $X$. Let $k \in \Z_{\geq 1}$. Assume that for every map
$f \colon Y \rightarrow X$ such that $Y$ is a connected curve
and $f_{*}[Y]=k \beta$, there exists a curve $C\subset X$ of class $\beta$ and a map $g \colon Y \rightarrow C$ such that 
$f$ is the composition of $g$ with the inclusion 
$C \subset X$. Assume further that 
$\beta \cdot D_1=\beta \cdot D_2=1$, and that 
$\eta^1$, $\eta^2$ are both the trivial $1$-part partition of $k$.
\begin{itemize}
\item[(i)] If $\delta^1_1=1 \in H^0(D^1,\Z)$ and $\delta^2_1=1 \in H^0(D_2,\Z)$, then, denoting 
 $\alpha_1 \in H^4(X,\Z)$ be the cohomology class Poincaré dual to a point, we have
\begin{equation} \label{eq_lem_van_2}\langle \eta^2,\delta^2|\, (-1)^g \lambda_g; \alpha_1|\,\eta^1,\delta^1 \rangle_{g,k\beta}^{X/D_1 \cup D_2}  =\begin{cases}1 \,\,\,\mathrm{if}\, g=0\\
0 \,\,\,\mathrm{else}\,.
\end{cases} 
\end{equation}
\item[(ii)]If  $\delta^1_1=1 \in H^0(D^1,\Z)$ and  $\delta^2_1 \in H^2(D_2,\Z)$ is the class Poincaré dual to a point,  
then
\begin{equation} \label{eq_lem_van_3}\langle \eta^2,\delta^2|\, (-1)^g \lambda_g|\,\eta^1,\delta^1 \rangle_{g,k\beta}^{X/D_1 \cup D_2}  =\begin{cases}\frac{1}{k} \,\,\,\mathrm{if}\, g=0\\
0 \,\,\,\mathrm{else}\,.
\end{cases} 
\end{equation}
\end{itemize}
\end{lem}
\begin{proof}
The vanishings for $g>0$ follows from Lemma \ref{lem_vanishing} for (i) and from a parallel proof for (ii). 
As in the proof of Lemma \ref{lem_vanishing}, we fix a general point $p \in X$ and let $C\simeq \PP^1 \subset X$ be the unique curve of class $\beta$ passing through $p$. 
There exists a unique degree $k$ map 
$f \colon Y \simeq \PP^1 \rightarrow C \simeq \PP^1$ fully ramified over $D_1 \cap C$ and $D_2 \cap C$. The automorphism group of $f$ is 
$\Z/k$ and so we obtain 
$\langle \eta^2,\delta^2|\,\eta^1,\delta^1 \rangle_{0,k\beta}^{X/D_1 \cup D_2}=1/k$ in (ii).
For (i), the point insertion at the interior marking kills all the non-trivial automorphisms and so $\langle \eta^2,\delta^2| \alpha_1|\,\eta^1,\delta^1 \rangle_{0,k\beta}^{X/D_1 \cup D_2}=1$.  
\end{proof}

\section{The key calculation} \label{section_computation}

In this section, we prove our key technical result, Theorem \ref{key_thm}, 
which computes explicitly a class of relative Gromov-Witten invariants 
$N_{g,\rel}^{\mu\nu\rho\sigma}$
of blown-up Hirzebruch surfaces
$\F_k^{\rho\sigma}$. In the following \S 
\ref{section_main_result}, we only use 
the relative Gromov-Witten invariants 
$N_{g,\rel}^{\mu\nu\emptyset\emptyset}$
of the Hirzebruch surfaces
$\F_k^{\emptyset \emptyset}=\F_k$ without additional blow-ups. However, our strategy of calculation, based on on the idea of trading 
relative conditions for blow-ups, requires us to consider the more general invariants $N_{g,\rel}^{\mu\nu\rho\sigma}$ even if one is ultimately only interested in the invariants $N_{g,\rel}^{\mu\nu\emptyset\emptyset}$ .

\subsection{Blown-up Hirzebruch surfaces}

We fix two nonnegative integers $k$ and $d$. As in \S \ref{ex_hirzebruch}, we consider the Hirzebruch surface $\F_k$, with its toric divisors 
$D_k$, $D_{-k}$ such that $D_k^2=k$, 
$D_{-k}^2=-k$, and we denote by $F$ the class of a $\PP^1$-fiber of $p \colon \F_k \rightarrow \PP^1$.
Let $\mu
=(\mu_j)_{1\leq j\leq \ell(\mu)}$, $\nu=(\nu_j)_{1\leq j\leq \ell(\nu)}$, 
$\rho=(\rho_j)_{1\leq j\leq \ell(\rho)}$, 
$\sigma=(\sigma_j)_{1\leq j\leq \ell(\sigma)}$
be four partitions of sums 
$|\mu|$,
$|\rho|$,
$|\nu|$,
$|\sigma|$,
such that
\begin{equation}\label{eq_sum_partition}
|\mu|+|\rho|=d \,\,\,\,
\text{and}\,\,\,\,|\nu|+|\sigma|=d+k\,.
\end{equation}

Let $\pi \colon \F_k^{\rho \sigma} \rightarrow \F_k$ be a 
blow-up of
$\F_k$ at $\ell(\rho)$ distinct points on $D_{-k}$
and $\ell(\sigma)$ distinct points on $D_k$. We denote by $E_j$, $1 \leq j \leq \ell(\rho)$,
and $F_j$, $1 \leq j \leq \ell(\sigma)$, the corresponding exceptional divisors. We still denote by $D_k$ and $D_{-k}$ the strict transforms in 
$\F_k^{\rho \sigma}$ of the divisors $D_k$ and $D_{-k}$ of $\F_k$, and by $F$ the pullback to $\F_k^{\rho \sigma}$ of the fiber class 
$F$ of $\F_k$.
We define the class $\beta_d^{\rho \sigma} \in H_2(\F_k^{\rho \sigma},\Z)$ 
by 
\begin{equation}
    \beta_d^{\rho \sigma} \coloneqq \pi^{*}(D_k+dF) -\sum_{j=1}^{\ell(\rho)} \rho_j E_j 
    -\sum_{j=1}^{\ell(\sigma)} \sigma_j F_j\,.
\end{equation}
We have the following intersection numbers
\begin{equation}
    \beta_d^{\rho \sigma} \cdot E_j =\rho_j \,, \,\,\,\,
    \beta_d^{\rho \sigma} \cdot F_j = \sigma_j \,,\,\,\,\, \end{equation}
\begin{equation}    \label{eq_intersection_numbers_2}
\beta_d^{\rho \sigma} \cdot D_{-k}=|\mu|=d-|\rho|\,,\,\,\,\,
\beta_d^{\rho \sigma} \cdot D_k=|\nu|=d+k-|\sigma|\,,\,\,\,\,
\beta_d^{\rho \sigma} \cdot F=1\,.
\end{equation}

\begin{figure}
\center{\scalebox{.3}{\input{1.pspdftex}}}
\caption{Illustration of the curves considered in the definition of the invariants $N_{g,\rel}^{\mu\nu\rho\sigma}$.}
\label{Fig:Nmunurhosigma}
\end{figure}

\subsection{Definition of the invariants $N_{g, \rel}^{\mu \nu \rho \sigma}$}
\label{def_invariants}

We define the relative Gromov-Witten invariants $N_{g,\rel}^{\mu \nu \rho \sigma}$ of 
$\F_k^{\rho\sigma}/D_k \cup D_{-k}$ by 
\begin{equation} \label{eq_gw_blowup}
    N_{g,\rel}^{\mu \nu \rho \sigma} \coloneqq \langle \mu,\delta^2|\, (-1)^g \lambda_g; \alpha_1\,|\,\nu,\delta^1 \rangle_{g,\beta_d^{\rho \sigma}}^{\F_k^{\rho\sigma}/D_k \cup D_{-k}}   \,,
\end{equation}
where we apply the general definition \eqref{eq_gw} of relative Gromov-Witten invariants to $X=\F_k^{\rho\sigma}$, 
$D_1=D_k$, $D_2=D_{-k}$, $\beta=\beta_d^{\rho\sigma}$,
$\eta^1=\mu$, $\eta^2=\nu$ and $r=1$. Note that $\mu$ and $\nu$ are indeed partitions of $\beta_d^{\rho\sigma} \cdot D_{-k}$
and $\beta_d^{\rho\sigma} \cdot D_{k}$ by \eqref{eq_intersection_numbers_2}. Moreover, in 
\eqref{eq_gw_blowup},
\begin{itemize}
    \item[(i)] we have $\delta^1=(\delta^1_j)_{1\leq j\leq \ell(\nu)}$, where $\delta^1_j \in H^2(D_k,\Z)$ is the cohomology class on $D_k$ Poincaré dual to a point for all $j$.
    \item[(ii)] we have $\delta^2=(\delta^2_j)_{1\leq j\leq \ell(\mu)}$, where $\delta^2_j \in H^2(D_{-k},\Z)$ is the cohomology class on $D_{-k}$ Poincaré dual to a point for all $j$.
    \item[(iii)] the class $\alpha_1 \in H^4(\F_k^{\rho \sigma})$ inserted at the single interior marked point is the cohomology class on $\F_k^{\rho\sigma}$ Poincaré dual to a point.
    \item[(iv)] the class $\gamma$ in \eqref{eq_gw} on the moduli space of relative stable maps is taken to be 
    $(-1)^g \lambda_g$, where $\lambda_g$ is the top lambda class, as in \eqref{eq_lambda}.
\end{itemize}

In other words,  $N_{g,\rel}^{\mu \nu \rho \sigma}$ is a virtual count of genus $g$ class $\beta_d^{\rho\sigma}$ curves
in $\F_k^{\rho\sigma}$, with contact orders along $D_{-k}$
(resp.\ $D_k$) given by $\mu$ (resp.\ $\nu$), and with fixed  
 position of the contacts points with $D_k \cup D_{-k}$
 and of a single interior marked point (see Figure \ref{Fig:Nmunurhosigma}).

\subsection{Empty partitions $\mu$ and $\nu$} \label{section_base_case}
In this section, we compute the invariants $N_{g,\rel}^{\mu\nu\rho\sigma}$ in the case where the partitions $\mu$ and $\nu$ are empty.

\begin{lem} \label{prop_base_case}
Assume that $\mu=\nu=\emptyset$. Then
\begin{equation}
     \sum_{g \geq 0} N_{g, \rel}^{\emptyset \emptyset \rho \sigma} u^{2g-2} 
     = \begin{cases}
         u^{-2} &\mathrm{if}\,\, \mathrm{all}\,\, \mathrm{the}\,\, \mathrm{parts}\,\, \mathrm{of}\, \rho\, \mathrm{and} \, \sigma \,\mathrm{are}\,\, \mathrm{equal}\,\, \mathrm{to}\,\, 1   \\
         0 & \mathrm{else}\,.
         \end{cases}
\end{equation}
\end{lem}

\begin{proof}
For $\mu =\nu=\emptyset$, we have 
$\beta_d^{\rho\sigma}\cdot D_{-k}=\beta_d^{\rho \sigma}\cdot D_k=0$
by \eqref{eq_intersection_numbers_2}.
The canonical class of $\F_k^{\rho\sigma}$
is $K_{\F_k^{\rho\sigma}}=-(D_k+D_{-k}+2F)$
and so it follows from \eqref{eq_beta_hirzebruch}
that \begin{equation} \label{eq_K}\beta_d^{\rho \sigma} \cdot K_{\F_k^{\rho \sigma}} = -2 \,.
\end{equation}
On the other hand, using that for $\mu=\nu=\emptyset$
we have $\ell(\rho)=k$ and $\ell(\sigma)=d+k$ by 
\eqref{eq_sum_partition} and so
\begin{equation}\label{eq_beta_carre}\beta_d^{\rho \sigma} \cdot \beta_d^{\rho\sigma} =k+2d - \sum_{j=1}^d \rho_j^2
-\sum_{j=1}^{d+k} \sigma_j^2
= -\sum_{j=1}^d \rho_j(\rho_j -1) - \sum_{j=1}^{d+k} \sigma_j (\sigma_j -1) \,.\end{equation}
Using the adjunction formula and 
\eqref{eq_K}-\eqref{eq_beta_carre}, a curve of class $\beta_d^{\rho \sigma}$
has arithmetic genus 
\begin{equation}
\frac{\beta_d^{\rho\sigma}\cdot (\beta_d^{\rho\sigma}+K_{\F_k^{\rho \sigma}})}{2}+1
=
- \sum_{j=1}^d \frac{ \rho_j( \rho_j -1)}{2} - \sum_{j=1}^{d+k} \frac{\sigma_j (\sigma_j -1)}{2}\,,\end{equation}
which is equal to zero if all the parts of $\rho$ and $\sigma$
are equal to $1$, and is negative else.
In particular, if one of the parts of $\rho$
or $\sigma$ is strictly greater than $1$, then the moduli space of relative stable maps used to define 
$N_{g, \rel}^{\emptyset \emptyset \rho \sigma}$ is empty and so
$N_{g, \rel}^{\emptyset \emptyset \rho \sigma}=0$.

If 
all the parts of $\rho$ and $\sigma$ are equal to $1$, then 
$\beta_d^{\rho \sigma} \cdot \beta^{\rho \sigma}=0$.
Furthermore, curves of class $\beta_d^{\rho\sigma}$ in 
$\F_k^{\rho\sigma}$ are strict transforms of curves 
in $\F_k =\PP(\cO_{\PP^1}\oplus\cO_{\PP^1}(k))$
that are graphs of rational sections of 
$\cO_{\PP^1}(k)$ of the form 
$\frac{s_{d+k}}{s_d}$, where the zeros of $s_{d+k}$ 
(resp.\ $s_d$) are the points that we blow-up on 
$D_k$ (resp.\ $D_{-k}$) to define $\F_k^{\rho\sigma}$ from $\F_k$. These rational sections are uniquely determined up to a multiplicative constant and so the linear system of curves of class $\beta_d^{\rho \sigma}$ is $1$-dimensional and consists of smooth rational curves. Thus, there is a unique curve of class $\beta_d^{\rho,\sigma}$ passing through a given general point and so $N_{0, \rel}^{\emptyset \emptyset \rho \sigma}=1$.
On the other hand, the assumptions of Lemma \ref{lem_vanishing} are satisfied and so $N_{g, \rel}^{\emptyset \emptyset \rho \sigma}=0$ if $g>0$. 

\end{proof}

\subsection{No horizontal component in bubbles} \label{section_no_horizontal}
In this section, we prove Lemma \ref{key_lemma}, a technical result on the form of the relative stable maps that contribute to the relative Gromov-Witten invariants $N_{g, \rel}^{\mu \nu \rho \sigma}$.

\begin{figure}
\center{\scalebox{.4}{\input{F1.pspdftex}}}
\caption{The expanded degeneration 
$\F_k^{\rho\sigma}[n_1,n_2]$ of $F_k^{\rho\sigma}$, with a chain of 
$n_1$ bubbles attached to $D_{-k}$ and a chain of $n_2$ bubbles attached to $D_k$.}
\label{Fig:bubbles}
\end{figure}

Recall from \cite{MR1882667} that a relative stable map to 
$\F_k^{\rho\sigma}/D_k \cup D_{-k}$ is really a stable map 
with target an expanded degeneration $\F_k^{\rho\sigma}[n_1,n_2]$
of $\F_k^{\rho\sigma}$ along $D_k$ and $D_{-k}$ for some $n_1$ and $n_2$. More precisely, $\F_k^{\rho\sigma}[n_1,n_2]$ is obtained from 
$\F_k^{\rho\sigma}$ by $n_1$ successive degenerations to the normal cone of $D_{-k}$ and $n_2$ successive degenerations to the normal cone of $D_{k}$. Concretely,  $\F_k^{\rho\sigma}[n_1,n_2]$ is obtained from 
$\F_k^{\rho\sigma}$ by gluing a length $n_1$ chain of bubbles $B_1,\dots,B_{n_1}$ along $D_{-k}$, where each bubble $B_j$ is isomorphic to a 
$\PP^1$-bundle over $D_{-k}$, and a length $n_2$
chain of bubbles 
$B_1',\dots,B_{n_2}'$ along $D_k$, where each bubble $B_j'$ is isomorphic to a $\PP^1$-bundle over $D_k$.
Each $\PP^1$-bundle $B_j \rightarrow D_{-k}$
(resp.\ $B_j' \rightarrow D_k$) admits two 
natural sections $D_{-k,0}^{(j)}$, 
$D_{-k,\infty}^{(j)}$ (resp.\ 
$D_{k,0}^{(j-1)}$, 
$D_{k,\infty}^{(j)}$). The bubbles $B_{j-1}$ and 
$B_j$ (resp.\ $B_{j-1}'$ and $B_j'$) are 
transversally glued together along $D_{-k,\infty}^{(j-1)} \simeq D_{-k,0}^{(j)}$ 
(resp.\  $D_{k,\infty}^{(j-1)} \simeq D_{k,0}^{(j)}$), and the tangency conditions at the relative markings are imposed along the divisors 
$D_{-k,\infty}^{(n_1)}$ and $D_{k,\infty}^{(n_2)}$ (see Figure \ref{Fig:bubbles}). 

A relative stable map 
$f \colon C \rightarrow \F_k^{\rho\sigma}[n_1,n_2]$ is also required to be \emph{pre-deformable} (\cite{MR1882667}), that is, 
\begin{itemize}
\item[(i)] $f(C)$ does not contain any of the divisors $D_{-k,\infty}^{(j-1)} \simeq D_{-k,0}^{(j)}$ and $D_{k,\infty}^{(j-1)} \simeq D_{k,0}^{(j)}$.
\item[(ii)]If there exists a point $x \in C$ such that $f(C) \in D_{-k,\infty}^{(j-1)} \simeq D_{-k,0}^{(j)}$, then $x$ is a node of $C$, where two irreducible components 
$C_{j-1}$ and $C_j$ of $C$ meet. Moreover, we have
$f(C_{j-1}) \subset B_{j-1}$, 
$f(C_j)\subset B_j$, and the contact order of 
$f(C_{j-1})$ along $D_{-k,\infty}^{(j-1)} \simeq D_{-k,0}^{(j)}$ at the point $x$ is equal to the contact order of $f(C_j)$ along $D_{-k,\infty}^{(j-1)} \simeq D_{-k,0}^{(j)}$ at the point $x$.
\item[(iii)]Same as (ii) with $D_{-k}$ replaced by $D_k$.
\end{itemize}

There is a natural morphism 
$\pr \colon \F_k^{\rho \sigma}[n_1,n_2]
\rightarrow \F_k^{\rho \sigma}$ that contracts the two chains of bubbles onto $D_{-k}$ and $D_k$
respectively.
We denote by 
$\tilde{p} \colon \F_k^{\rho\sigma} 
\rightarrow \PP^1$ the composition of the blow-up 
$\pi \colon \F_k^{\rho\sigma} \rightarrow \F_k$ with the 
$\PP^1$-fibration $p\colon\F_k \rightarrow \PP^1$.
Finally, we denote by $\widetilde{\pr} \colon 
\F_k^{\rho \sigma}[n_1,n_2] \rightarrow \PP^1$
the composition of $\pr \colon \F_k^{\rho \sigma}[n_1,n_2]
\rightarrow \F_k^{\rho \sigma}$
with $\tilde{p} \colon \F_k^{\rho\sigma} 
\rightarrow \PP^1$.

\begin{lem} \label{key_lemma}
Let 
\begin{equation} f \colon (C,x_1,y_1,\dots,y_{\ell(\mu)},z_1,\dots,z_{\ell(\nu)})  \rightarrow \F_k^{\rho \sigma}[n_1,n_2]
\end{equation}
be a relative stable map 
defining a point of 
\begin{equation} 
\overline{M}_{g,1}(\F_k^{\rho\sigma}/D_k\cup D_{-k},\beta_d^{\rho\sigma},\mu,\nu) \,.
\end{equation}
Assume that the points  
$\widetilde{\pr}(f(x_1))$, $\widetilde{\pr}(f(y_j))$, $\widetilde{\pr}(f(z_l))$,
$\tilde{p}(E_m)$, 
$\tilde{p}(F_n)$ are all distinct on $\PP^1$.
Then the curve $(\pr \circ f)(C)$ in $\F_k^{\rho \sigma}$ does not contain 
$D_{-k}$ or $D_{k}$. In other words, the components of $C$ mapped to bubble 
components of $\F_k^{\rho \sigma}[n_1,n_2]$ are mapped onto $\PP^1$-fibers of the bubbles.
\end{lem}

\begin{proof}
By \eqref{eq_intersection_numbers_2}, we have $\beta_d^{\rho \sigma} \cdot F =1$ and so 
$(\pr \circ f)(C)$ contains at most one copy of $D_{-k}$ or $D_k$. 
Assume by contradiction that $(\pr \circ f)(C)$ contains one copy of $D_{-k}$, that is, 
that there exists a bubble $B_j$ and an irreducible component of $C$ mapped to $B_j$ and whose image is not contained in a $\PP^1$-fiber of $B_j$. 
Then, denoting by $C_0$ the union of components of $C$ mapped to $\F_k^{\rho\sigma}$ by $f$,  $f(C_0)$ is contained in a union of fibers of 
$\tilde{p} \colon \F_k^{\rho \sigma} \rightarrow \PP^1$. As we are assuming that 
$\widetilde{\pr}(f(x_1))$ is distinct from the points 
$\tilde{p}(E_m)$ and 
$\tilde{p}(F_n)$, the image by $f$ of the connected component $C_{0,x_1}$ of $C_0$
containing the interior marked point $x_1$ is a $\PP^1$-fiber of $\tilde{p} \colon \F_k^{\rho \sigma} \rightarrow \PP^1$
and so intersects $D_{k}$ (resp.\ $D_{-k}$) at a point $\alpha$ such that $\widetilde{\pr}(\alpha)
=\widetilde{\pr}(f(x_1))$.

The pre-deformability condition implies that there exists an 
irreducible component $C_1$ of $C$ mapping non-trivially to the bubble $B_1'$ and intersecting 
$D_{k,0}$ at $\alpha$.
As we are assuming that $f(C)$ does not contain a copy of $D_k$, $f(C_1)$ is the $\PP^1$-fiber of 
$B_1$ containing $\alpha$. 
Iterating the argument, we obtain that there exists an irreducible component $C_{n_2}$ of $C$ 
such that $f(C_{n_2})$ is a $\PP^1$-fiber 
of the last bubble $B_{n_2}'$ and 
$\widetilde{\pr}(f(C_{n_2}))=\widetilde{\pr}(f(x_1))$. In particular, there exists a point 
$\beta \in C_{n_2}$ such that $f(\beta) \in D_{k,\infty}^{(n_2)}$. But, by definition of relative stable maps, the only points of $C$ mapped to $D_{k,\infty}^{(n_2)}$ are the relative markings 
 $z_j$. By our assumption, we have 
 $\widetilde{\pr}(f(x_1)) \neq \widetilde{\pr}(f(z_j))$ 
 and so $\beta \neq z_j$ for every $j$, and we obtain a contradiction.

The argument when
$(\pr \circ f)(C)$ contains one copy of $D_{k}$ 
is identical after exchanging the roles of 
$D_k$ and $D_{-k}$.
\end{proof}

\subsection{Calculation of the invariants $N_{g,\rel}^{\mu\nu\rho\sigma}$}
\label{statement}

In this section, we prove Theorem \ref{key_thm} computing the invariants 
$N_{g,\rel}^{\mu\nu\rho\sigma}$. The key geometric argument is contained in the proof of Lemma \ref{lem_deg_key} where we use a degeneration argument to trade 
tangency conditions for blowups. This trick to exchange tangency conditions and blow-ups has been used since the early days of Gromov-Witten theory, see for example \cite{gathmann1996counting}. For another examples of application of this technique closely related to the present paper, we refer to \cite{MR2667135, bousseau2018quantum_tropical, bousseau2018example}. 

We start by introducing some notations about partitions that will be useful
to formulate the proof by induction of Theorem \ref{key_thm}.
If $\mu$ is a partition, we denote $\max(\mu)$ the greatest value attained by a part of $\mu$, and $\Nmax(\mu)$ the number of parts of $\mu$ attaining this maximum value. 
If $(\mu, \nu)$ is a pair of partitions, we denote $\max(\mu, \nu)$
for $\max(\max(\mu),\max(\nu))$, that is, the greatest value 
attained by a part of $\mu$ or a part of $\nu$, and 
$\Nmax(\mu, \nu)$ the number of parts of 
$\mu$ and $\nu$ attaining this maximum value.

If $\tau$ is a partition of $\max(\mu)$, we denote $\hat{\mu} \cup \tau$ the partition
of $\tau$ whose set of parts is the union of the set of parts of $\hat{\mu}$
and of the set of parts of $\tau$.
If $\tau$ is not the trivial $1$-part partition of 
$\max(\mu)$, we have
\begin{equation} (\max(\mu \cup \tau, \nu), \Nmax(\mu \cup \tau, \nu)) < (\max(\mu,\nu), \Nmax(\mu,\nu))\,.\end{equation}
If $\tau$ is the trivial $1$-part partition of $\max(\mu)$, we have
$\hat{\mu} \cup \tau = \mu$. 

Let $\mu$, $\nu$, $\rho$, $\sigma$ be four partitions.
Let $\hat{\rho}$ be the partition obtained from $\rho$ by adding one part equal to $\max(\mu)$. We denote $\hat{\mu}$ the partition of $\mu$ obtained 
from $\mu$ by removing one part equal to $\max(\mu)$.

\begin{lem} \label{lem_deg_key}
\begin{equation}\label{eq_deg_key}
\sum_{g \geq 0} N_{g,\rel}^{\hat{\mu} \nu \hat{\rho} \sigma}
u^{2g-2+\ell(\mu)+\ell(\nu)} \end{equation}
\begin{equation*}
= \sum_{\tau \vdash \max(\mu)}
\left( \sum_{g \geq 0} 
N_{g, \rel}^{(\hat{\mu}\cup \tau) \nu \rho \sigma}
u^{2g-2+\ell(\hat{\mu} \cup \tau)+\ell(\nu)}
\right)
\prod_{\ell \geq 1}
\frac{\ell^{m_\ell(\tau)}}{m_\ell(\tau) !} 
\left(\frac{(-1)^{\ell -1}}{\ell}
\frac{1}{2 \sin \left( \frac{\ell u}{2}\right)}
\right)^{m_\ell(\tau)} \,,
\end{equation*}
where we sum over the partitions $\tau=(\tau_j)_{1\leq j\leq \ell(\tau)}$ of 
$\max(\mu)$ and $m_\ell(\tau)$ is the number of parts of $\tau$ equal to
$\ell$.
\end{lem}

\begin{proof}
The proof is an application of the degeneration formula in Gromov-Witten theory to a specific degeneration of  
$\F_{k}^{\hat{\rho} \sigma}$.

Let $X$ be the degeneration of $\F_k^{\rho \sigma}$
to the normal cone of $D_{-k}$, that is, the blow-up of 
$D_{-k} \times \{0\}$ in $\F_k^{\rho \sigma} \times \A^1$. 
Let $\pi \colon  X \rightarrow \A^1$ be the natural projection.
The special fiber $\pi^{-1}(0)$ has two 
irreducible components 
$\F_{k}^{\rho \sigma}$ and $\PP_{-k}$.
Here $\PP_{-k}$ is a $\PP^1$-bundle over $D_{-k}$, with two natural sections $D_{-k,0}$ and 
$D_{-k,\infty}$. In $\pi^{-1}(0)$,
the divisor $D_{-k}$ of 
$\F_k^{\rho \sigma}$
is transversally glued with the divisor 
$D_{-k,0}$ of $\PP_{-k}$.
Let $s$ be a section of $\pi$ such that 
for every $t \neq 0$, 
$s(t) \in D_{-k}$, away from
$D_{-k} \cap E_j$
for all $1 \leq j \leq \ell(\rho)$, and such that $s(0) \in 
(D_{-k})_\infty$.
We blow-up the image of $s$ in $X$
to obtain a new family 
$\tilde{\pi} \colon \tilde{X}
\rightarrow \A^1$.
For $t \neq 0$, we identify 
$\tilde{\pi}^{-1}(t)$ with 
$\F_{k}^{\hat{\rho} \sigma}$.
The special fiber 
$\tilde{\pi}^{-1}(0)$
has two irreducible components:
$\F_k^{\rho \sigma}$ and 
$\tilde{\PP}_{-k}$ glued along the divisor $D_{-k}$, where 
$\tilde{\PP}_{-k}$
is the blow-up of 
$\PP_{-k}$ at the point $s(0)$. We denote by $C\simeq \PP^1 \subset
\tilde{\PP}_{-k}$ the $(-1)$-curve which is the strict transform of the 
$\PP^1$-fiber of $\PP_{-k} \rightarrow D_{-k}$ passing through 
$s(0)$.

We would like to compute the relative Gromov-Witten invariant
$N_{g,\rel}^{\hat{\mu} \nu \hat{\rho} \sigma}$ of $\F_{k}^{\hat{\rho} \sigma}$
using the degeneration $\tilde{\pi} \colon \tilde{X} \rightarrow \A^1$.
A priori, the degeneration formula of \cite{MR1938113} cannot be used to study the degeneration of a relative problem. But in the present situation, 
Lemma \ref{key_lemma} guarantees that the various 
relative conditions along $D_{-k}$ and $D_{k}$ never interact in a non-trivial way. 
It follows that the degeneration formula of \cite{MR1938113} can actually be applied to this case. 
The degeneration formula takes the form 
\begin{equation}\label{eq_deg} N_{g,\rel}^{\hat{\mu}} = \sum_{\Gamma} \frac{\prod_E w_E}{|\Aut(\Gamma)|} \prod_V N_V \,,
\end{equation}
where the sum is over decorated weighted bipartite graphs $\Gamma$
describing dual graphs of curves in the special fiber $\tilde{\pi}^{-1}(0)= \F_k^{\rho\sigma} \cup \tilde{\PP}_{-k}$:
\begin{itemize}
    \item[(i)] Every vertex $V$ of $\Gamma$ is either of type $\F_k^{\rho\sigma}$ or of type $\tilde{\PP}_{-k}$, corresponding to a curve component mapping either to 
    $\F_k^{\rho\sigma}$ or to $\tilde{\PP}_{-k}$. Every vertex $V$ is also decorated by a genus $g_V$ and a curve class $\beta_V$.
    \item[(ii)] Every edge $E$ of $\Gamma$ connects a vertex of type $\F_k^{\rho\sigma}$ with a vertex of type $\tilde{\PP}_{-k}$, and has a weight $w_E$.
    \item[(iii)] Half-edges $(V,E)$ of $\Gamma$, that is, pairs of a vertex $V$ and of an incident edge $E$, is decorated by a cohomology class $c_{(V,E)} \in\{1,p\}$, where $1 \in H^0(D_{-k},\Z)$ and $p\in H^2(D_{-k},\Z)$ is Poincaré dual to a point. For every edge $E$, there is exactly one vertex $V$
    incident to $E$ such that $c_{(V,E)}=1$, and one vertex $V'$ incident to $E$ such that $c_{(V',E)}=p$. These cohomology classes come from the insertion of the class 
    $1 \times p+p\time 1$ of the diagonal $D_{-k} \simeq \PP^1 \subset D_{-k} \times D_{-k} \simeq \PP^1 \times \PP^1$ in the degeneration formula  \cite{MR1938113}. 
\end{itemize}
Moreover, the contribution $N_V$ of each vertex $V$ is a relative Gromov-Witten invariant defined by the type of $V$ and the weights and cohomological decorations of the edges incident to $V$.

As the definition \ref{eq_gw_blowup} of $N_{g,\rel}^{\hat{\mu}}$ includes the insertion of the class $(-1)^g \lambda_g$, it follows from Lemma \ref{lem_gluing1}
that only graphs $\Gamma$ of genus $0$ can have a non-zero contribution and that the definition of $N_V$ includes the insertion of the class $(-1)^{g_V} \lambda_{g_V}$. Let $V$ be a vertex of type $\tilde{\PP}_{-k}$ of such graph $\Gamma$. Then, the class $\beta_V$ is necessarily a multiple $k_V[C]$ of the class of the $(-1)$-curve $C \subset \tilde{\PP}_{-k}$. Let $\ell_V$ be the number of edges incident to $V$ and $\mu_V$ the partition whose parts are the weights 
$w_E$ of edges incident to $V$. As the curve $C$ is rigid in 
$\tilde{\PP}_{-k}$, every stable map of class $k_V[C]$ factors through 
$C$. Moreover, $N_V$ is non-vanishing only if $c_{(V,E)}=1$ for every $E$ and then
\begin{equation} \label{eq_N_V_key}
    N_V = \int_{[M_V]^{\virt}} (-1)^{g_V} \lambda_{g_V} e(R^1\pi_V f_V^{*} \cO(-1)) \,,
\end{equation}
where $M_V$ is the moduli stack of genus $g_V$ degree $k_V$ relative stable maps to $C \simeq \PP^1$ relative to a point $\infty \in \PP^1$ and with contact orders $\mu_V$. Moreover,
$\pi_V \colon C_V \rightarrow M_V$ is the universal curve, $f_V \colon C_V \rightarrow \PP^1$ is the universal map and the Euler class insertion 
$e(R^1\pi_V f_V^{*} \cO(-1))$ is the difference between the perfect obstruction theories for curves mapping to the surface $\tilde{\PP}_{-k}$ and for curves mapping to the curve $C \simeq \PP^1$ with normal bundle $\cO(-1)$ in 
$\tilde{\PP}_{-k}$. The virtual dimension of $M_V$ is $2g_V-2+k_V+\ell_V$, whereas the integrand in \eqref{eq_N_V_key} has complex degree $2g-1+k_V$, and so $N_V=0$ unless $\ell_V=1$. If $\ell_V=1$, then $N_V$ is the coefficient of 
$u^{2g-1}$ in 
\begin{equation} \frac{(-1)^{k_V -1}}{k_V}
\frac{1}{2 \sin \left( \frac{k_V u}{2}\right)}\end{equation} 
by \cite[Theorem 5.1]{MR2115262}.

Therefore, it is enough in \eqref{eq_deg} to sum over genus $0$ graphs $\Gamma$
such that $\ell_V=1$ for every $V$ of type $\tilde{\PP}_{-k}$.
For such graph $\Gamma$, as $\Gamma$ is connected, there exists a unique vertex $V_0$ of type 
$\F_k^{\rho\sigma}$ and 
$N_{V_0}= N_{g_{V_0}, \rel}^{(\hat{\mu}\cup \tau) \nu \rho \sigma}$
where $\tau$ is the partition of $\max(\mu)$ whose parts are $k_V$
for $V$ of type $\tilde{\PP}_{-k}$. Hence, \eqref{eq_deg} reduces to 
\eqref{eq_deg_key}.
\end{proof}

\begin{thm} \label{key_thm}
For every partitions $\mu$, $\nu$, $\rho$,
$\sigma$ as in \eqref{eq_sum_partition}, the relative Gromov-Witten invariants 
$N_{g,\rel}^{\mu\nu\rho\sigma}$
of $\F_k^{\rho\sigma}/D_k \cup D_{-k}$
defined in \eqref{eq_gw_blowup} are as follows.
\begin{itemize}
    \item[(i)] If all the parts of $\rho$ and $\sigma$ are equal to $1$, then
    \begin{equation} \label{eq_thm_key} \sum_{g \geq 0} N_{g, \rel}^{\mu \nu \rho \sigma} u^{2g-2+\ell(\mu)+\ell(\nu)}
= u^{-2}
\prod_{j=1}^{\ell(\mu)}\frac{1}{\mu_j} 2\sin\left(
\frac{\mu_j u}{2}\right)
\prod_{l=1}^{\ell(\nu)}
\frac{1}{\nu_l}2\sin\left(
\frac{\nu_l u}{2}\right)\,,
\end{equation}
that is, using the notation \eqref{eq_q_integer} for $q$-integers,
\begin{equation}\sum_{g \geq 0} N_{g, \rel}^{\mu \nu \rho \sigma} u^{2g+\ell(\mu)+\ell(\nu)}
=u^{-2}\left( 
(-i)(q^{\frac{1}{2}} -q^{-\frac{1}{2}})\right)^{\ell(\mu)+\ell(\nu)} 
\prod_{j=1}^{\ell(\mu)}\frac{[\mu_j]_q}{\mu_j} 
\prod_{l=1}^{\ell(\nu)}
\frac{[\nu_l]_q}{\nu_l}
\,, 
\end{equation}
where $q=e^{iu}$ in the right-hand side.
   \item[(ii)] If the parts of $\rho$ and $\sigma$ are not not all equal to
   $1$, then $N_{g, \rel}^{\mu \nu \rho \sigma}=0$
for all $g \geq 0$.
\end{itemize}
\end{thm}

\begin{proof}
We prove
Theorem \ref{key_thm} by induction on the pair
$(\max(\mu,\nu), \Nmax(\mu,\nu))$, where we use the lexicographic order
for pairs of nonnegative integers:
$(x,y) \leq (x',y')$ if $x \leq x'$, or $x=x'$ and 
$y \leq y'$. Concretely, at every step, we lower the 
number of times that the maximum value for parts of $\mu$ and $\nu$
is attained, and once this number of times is reduced to one, we reduce this maximum value. 
The base case of the induction is  $\mu= \nu
= \emptyset$ and the result is then given by Lemma 
\ref{prop_base_case}.

The remainder of the proof is the inductions step. Let $\mu$, $\nu$, $\rho$, $\sigma$ be four partitions. We assume that Theorem \ref{key_thm} holds for 
every partitions $\mu'$, $\nu'$, $\rho'$, $\sigma'$ with 
$(\max(\mu',\nu'), \Nmax(\mu',\nu')) < (\max(\mu,\nu), \Nmax(\mu,\nu))$. We want to show that 
Theorem \ref{key_thm} holds for $\mu$, $\nu$, $\rho$, $\sigma$.
Up to exchanging the roles of $\mu$ and $\nu$, we can assume that 
$\max(\mu,\nu)=\max(\mu)$, that is, $\max(\mu,\nu)$ is attained by a part of 
$\mu$. 

By Lemma \ref{lem_deg_key}, 
$N_{g,\rel}^{\hat{\mu} \nu \hat{\rho}}$
is expressed by \eqref{eq_deg_key}
in terms of the invariants
$
N_{g, \rel}^{(\hat{\mu}\cup \tau) \nu \rho \sigma}$
where $\tau=(\tau_j)_{1\leq j\leq \ell(\tau)}$ is a partition 
of $\max(\mu)$.
If $\tau$ is not the trivial $1$-part partition of 
$\max(\mu)$, we have
\begin{equation} (\max(\mu \cup \tau, \nu), \Nmax(\mu \cup \tau, \nu)) < (\max(\mu,\nu), \Nmax(\mu,\nu))\,,\end{equation}
and so by the induction hypothesis, we can apply Theorem \ref{key_thm} to
compute the invariants
$N_{g,\rel}^{(\hat{\mu}\cup \tau) \nu \rho \sigma}$.
Similarly, we have \begin{equation}(\max(\hat{\mu},\nu), \Nmax(\hat{\mu},\nu))<(\max(\mu,\nu),\Nmax(\mu,\nu))\,,\end{equation}
and so by the induction hypothesis, we can apply Theorem \ref{key_thm} 
to compute
$N_{g,\rel}^{\hat{\mu} \nu \hat{\rho} \sigma}$. 
If $\tau$ is the trivial $1$-part partition of $\max(\mu)$, we have 
$\hat{\mu} \cup \tau =\mu$ and so 
$N_{g,\rel}^{(\hat{\mu}\cup \tau) \nu \rho \sigma}
= N_{g,\rel}^{\mu \nu \rho \sigma}$.
Hence, it remains to prove that \eqref{eq_thm_key} is indeed implied
by \eqref{eq_deg_key} and the induction hypothesis.

If a part of $\rho$ or $\sigma$ is not equal to $1$,
then, by the induction hypothesis, we have 
$N_{g,\rel}^{\hat{\mu} \nu \hat{\rho} \sigma}=0$ for every $g \geq 0$, and 
$N_{g,\rel}^{(\hat{\mu}\cup \tau) \nu \rho \sigma}=0$ for every $g \geq 0$ and for every $\tau$ non-trivial partition of $\max(\mu)$. 
It follows from \eqref{eq_deg_key} that $N_{g,\rel}^{\mu \nu \rho \sigma}=0$ for every
$g \geq 0$.

So we can assume that all parts of $\rho$ and $\sigma$ are equal to $1$.
If $\max(\mu) \neq 1$, then a part of $\hat{\rho}$ is strictly greater than $1$, 
and so 
$N_{g,\rel}^{\hat{\mu} \nu \hat{\rho} \sigma}=0$.
By induction hypothesis, we have 
\begin{equation}\sum_{g \geq 0}N_{g, \rel}^{(\hat{\mu}\cup \tau) \nu \rho \sigma}
u^{2g-2+\ell(\hat{\mu}+m)+\ell(\nu)}=\end{equation}
\begin{equation*}u^{-2}\prod_{\ell \geq 1}\left(\frac{1}{\ell} 2\sin\left(
\frac{\ell u}{2}\right)\right)^{m_\ell(\hat{\mu})}
\left(\frac{1}{\ell}2\sin\left(
\frac{\ell u}{2}\right) \right)^{m_\ell(\nu)}
\left(\frac{1}{\ell}2\sin\left(
\frac{\ell u}{2}\right) \right)^{m_\ell(\tau)}\end{equation*}
for every non-trivial partition $\tau$ of $\max(\mu)$, where
$m_\ell(\mu)$ is the number of parts equal to $\ell$ in a partition $\mu$.
In order to prove \eqref{eq_thm_key},
it is enough by \eqref{eq_deg_key} to prove that
\begin{equation}\label{eq_identity} \sum_{\tau \vdash \max(\mu)} \prod_{\ell \geq 1} 
\frac{1}{m_\ell(\tau) !} \left( \frac{(-1)^{\ell -1}}{\ell} \right)^{m_\ell(\tau)}=0 \,.
\end{equation}
But the left-hand side of \eqref{eq_identity} is the coefficient of
$x^{\max(\mu)}$ in the power series expansion of the identity 
$\exp (\log(1+x))=1+x$, 
and so vanishes as we are assuming $\max(\mu)>1$.

If $\max(\mu)=1$, then all the parts of $\mu$ and $\nu$ are equal to $1$ and
\eqref{eq_deg_key} reduces to
\begin{equation} \sum_{g \geq 0} N_{g, \rel}^{\hat{\mu} \nu \hat{\rho} \sigma}
u^{2g-2+\ell(\mu)+\ell(\nu)}
=u^{-2}
\left(
\sum_{g \geq 0} N_{g, \rel}^{\mu \nu \rho \sigma} u^{2g+\ell(\mu)+\ell(\nu)}
\right)\frac{1}{2 \sin \left(\frac{u}{2}
\right)}\,.
\end{equation}
By induction hypothesis, we have  
\begin{equation} \sum_{g \geq 0}N_{g, \rel}^{\hat{\mu} \nu \hat{\rho} \sigma}u^{2g-2+\ell(\mu)+\ell(\nu)}
=u^{-2}
\left( 2\sin\left(
\frac{u}{2}\right)\right)^{m_1(\hat{\mu})}
\left(2\sin\left(
\frac{ u}{2}\right) \right)^{m_1(\nu)}\,,
\end{equation}
and so, using that $m_1(\mu) =m_1(\hat{\mu})+1$, we obtain
\begin{equation}\sum_{g \geq 0}N_{g, \rel}^{\mu \nu \rho \sigma}u^{2g-2+\ell(\mu)+\ell(\nu)}
=u^{-2}
\left( 2\sin\left(
\frac{u}{2}\right)\right)^{m_1(\mu)}
\left(2\sin\left(
\frac{ u}{2}\right) \right)^{m_1(\nu)}\,.
\end{equation}
This finishes the proof of Theorem \ref{key_thm}.
\end{proof}

\section{Main result: floor diagrams from degeneration} \label{section_main_result}

In \S \ref{preliminaries}, we prove two vanishing results in relative Gromov-Witten theory of Hirzebruch surfaces. 
In \S \ref{section_relative_gw}, we define the relative 
Gromov-Witten invariants $N_{g,\rel}^{\Delta,n}$ of $h$-transverse toric sufaces.
In \S \ref{section_degeneration}, for $\Delta=\Delta_{h,d}^{\F_k}$, we apply the degeneration formula in Gromov-Witten theory to express the invariants 
$N_{g,\rel}^{\Delta,n}$ in terms of the invariants 
$N_{g,\rel}^{\mu\nu\emptyset\emptyset}$ defined in \S \ref{section_computation}.
The unrefined, that is, 
$g=g_{\Delta,n}$, version of this degeneration argument can be found for example in the proof of Theorem 4.9 of \cite{cavalieri2017counting}, or, in the Fock space language, in \S 2.5 of \cite{cooper2017fock}. We adapt this degeneration argument to the refined, that is,
$g \geq g_{\Delta,n}$, case using the vanishing results proved in 
\S \ref{preliminaries}. Finally, we prove in \S \ref{section_main_result} our main result, Theorem \ref{main_thm_precise}, computing the invariants 
$N_{g,\rel}^{\Delta,n}$ for $\Delta=\Delta_d^{\PP^2}$ and 
$\Delta=\Delta_{h,d}^{\F_k}$ in terms of $q$-refined counts of floor diagrams.
The proof relies on the explicit calculation of the invariants 
$N_{g,\rel}^{\mu\nu\rho\sigma}$ given by Theorem \ref{key_thm}.

\subsection{Dimension constraints}
\label{preliminaries}
In this section, we prove two vanishing results for relative Gromov-Witten invariants of Hirzebruch surfaces. 
We use the notations introduced in 
\S \ref{section_relative_gw_general}
for relative Gromov-Witten invariants. 

\begin{lem}
\label{lem_dim_1}
Let $h, d\in \Z_{\geq 0}$, and $\beta =hD_k+dF \in H_2(\F_k,\Z)\setminus \{0\}$.
Let
$\mu=(\mu_j)_{1\leq j\leq \ell(\mu)}$, $\nu
=(\nu_j)_{1\leq j\leq \ell(\nu)}$ be two partitions of $\beta \cdot D_{-k}=d$ and 
$\beta \cdot D_k=d+kh$ respectively. 
Let $\delta^1=(\delta^1_j)_{1\leq j\leq \ell(\mu)}$ be 
$\ell(\mu)$ elements of $H^{*}(D_{-k},\Z)$ and 
$\delta^2=(\delta_j^2)_{1\leq j \ell(\nu)}$ be
$\ell(\nu)$ elements of $H^{*}(D_{k},\Z)$.
Assume that among these  
$\ell(\mu)+\ell(\nu)$ cohomology classes, 
$s$ of them are equal to $1$ and $\ell(\mu)+\ell(\nu)-s$ of them are Poincaré dual to a point. Then, for every $g \in \Z_{\geq 0}$ and $0\leq j \leq g$, we have 
\begin{equation}\label{eq_lem_pre_0} \langle \mu,\delta^2|\, (-1)^j \lambda_j|\,\nu,\delta^1 \rangle_{g,\beta}^{\F_k/D_k \cup D_{-k}} =0 \,,   
\end{equation}
unless the following conditions hold: 
\begin{itemize}
\item[(i)]$h=0$, that is $\beta=dF$. In such case, we have $\beta \cdot D_{-k}=\beta \cdot D_k = d$.
\item[(ii)]
 $\mu$ and $\nu$ are both the trivial 1-part partition $(d)$ of $d$, that is, $\ell(\mu)=\ell(\nu)=1$.
\item[(ii)] $s=1$, that is among the 
$\ell(\mu)+\ell(\nu)=2$ cohomology classes 
$\delta^1_1$ and $\delta^2_1$, exactly one of them is equal to $1$ and the other is Poincaré dual to a point. 
\item[(iv)] $g=j=0$.
\end{itemize}
If these conditions are satisfied, then 
\begin{equation}\label{eq_lem_pre} \langle \mu,\delta^2|\,\nu,\delta^1 \rangle_{0,\beta}^{\F_k/D_k \cup D_{-k}} =\frac{1}{d} \,.   
\end{equation}
\end{lem}

\begin{proof}
By \eqref{eq_vdim}, the virtual dimension of the moduli stack of relative stable maps used to define 
\begin{equation}\label{eq_gw_1} \langle \mu,\delta^2|\, (-1)^j \lambda_j|\,\nu,\delta^1 \rangle_{g,\beta}^{\F_k/D_k \cup D_{-k}}    
\end{equation}
is 
\begin{align}\label{eq_vdim_1}
&g-1+\beta \cdot (D_k+D_{-k}+2F)-(\beta \cdot D_{-k} -\ell(\mu))-(\beta \cdot D_k -\ell(\nu)) \\
&=g-1+2h+\ell(\mu)+\ell(\nu)\,.\nonumber
\end{align}
On the other hand, we integrate 
in \ref{eq_gw_1}
over the virtual dimension class a cohomology class of complex degree 
\begin{equation} \label{eq_insert_1}
j+\ell(\mu)+\ell(\nu)-s\,.
\end{equation}
\eqref{eq_gw_1} is $0$ unless 
\eqref{eq_vdim_1} = \eqref{eq_insert_1}, that is 
$2h+s+(g-j)=1$. As $h$, $s$ and $(g-j)$ are nonnegative integers, this is only possible if either $(h,s,j)=(0,0,g-1)$ or 
$(h,s,j)=(0,1,g)$. 

If $(h,s,j)=(0,0,g-1)$, then $\beta=d F$ and we fix the position of the $\ell(\mu)+\ell(\nu)$ contact points with 
$D_k \cup D_{-k}$. Note that $\ell(\mu)$, $\ell(\nu) \geq 1$ because the assumption $\beta \neq 0$ implies that $d>0$.
As $\ell(\mu)+\ell(\nu)\geq 2$, we can choose the position of two contact points in two different fibers of $p \colon \F_k \rightarrow \PP^1$. But a curve of class $\beta=dF$ is contained in a $\PP^1$-fiber of $p$, so the set of curves matching the constraints is empty and so \eqref{eq_gw_1} is zero.

If $(h,s,j)=(0,1,g)$, then $\beta=d F$ and we fix the position of $\ell(\mu)+\ell(\nu)-1$
of the $\ell(\mu)+\ell(\nu)$ contact points with 
$D_k \cup D_{-k}$.
If $\ell(\mu)+\ell(\nu)-1\geq 2$, we can choose the position of two contact points in two different fibers of $p$ and as a curve of class $\beta=dF$ is contained in a $\PP^1$-fiber of $p$, the set of curves matching the constraints is empty. Hence, \eqref{eq_gw_1} is still zero unless $\ell(\mu)=\ell(\nu)=1$. Finally, under the assumptions (i)-(iv), \eqref{eq_lem_pre} follows from 
Lemma \ref{lem_vanishing_2}(ii).
\end{proof}

\begin{lem} \label{lem_dim_2}
Let $h,d \in \Z_{\geq 0}$, and $\beta =hD_k+dF \in H_2(\F_k,\Z)\setminus \{0\}$.
Let
$\mu=(\mu_j)_{1\leq j\leq \ell(\mu)}$, $\nu
=(\nu_j)_{1\leq j\leq \ell(\nu)}$ be two partitions of $\beta \cdot D_{-k}=d$ and 
$\beta \cdot D_k=d+kh$ respectively. 
Let $\delta^1=(\delta^1_j)_{1\leq j\leq \ell(\mu)}$ be 
$\ell(\mu)$ elements of $H^{*}(D_{-k},\Z)$ and 
$\delta^2=(\delta_j^2)_{1\leq j \ell(\nu)}$ be
$\ell(\nu)$ elements of $H^{*}(D_{k},\Z)$.
Assume that among these  
$\ell(\mu)+\ell(\nu)$ cohomology classes, 
$s$ of them are equal to $1$ and $\ell(\mu)+\ell(\nu)-s$ of them are Poincaré dual to a point. Then, for every $g \in \Z_{\geq 0}$
and $0\leq j\leq g$, denoting by $\alpha_1\in H^4(X,Z)$ the class 
Poincaré dual to a point, we have 
\begin{equation} \langle \mu,\delta^2|\, (-1)^j \lambda_j; \alpha_1|\,\nu,\delta^1 \rangle_{g,\beta}^{\F_k/D_k \cup D_{-k}} =0 \,,   
\end{equation}
unless we are in one of the following two situations.
\begin{itemize}
\item[(i)]
$h=0$, that is $\beta=dF$, $\mu$ and $\nu$ are both the trivial 1-part partition $(d)$ of $d$, that is, $\ell(\mu)=\ell(\nu)=1$, $s=1$, that is both of the
$\ell(\mu)+\ell(\nu)=2$ cohomology classes 
$\delta^1_1$ and $\delta^2_1$ are equal to $1$, and $g=j=0$. In this case, we have 
\begin{equation} \label{eq_lem_pre_2}
\langle \mu,\delta^2|\, \alpha_1|\,\nu,\delta^1 \rangle_{0,\beta}^{\F_k/D_k \cup D_{-k}} =1 \,.
\end{equation}
\item[(ii)] $h=1$, that is 
$\beta=D_k+dF$, $s=0$, that is all of the 
$\ell(\mu)+\ell(\nu)$ cohomology classes 
$\delta^1_j$ and $\delta^2_j$ are Poincaré dual to a point, and $j=g$. In this case, we have 
\begin{equation} \label{eq_lem_pre_3}
\langle \mu,\delta^2|\, (-1)^g \lambda_g; \alpha_1|\,\nu,\delta^1 \rangle_{g,\beta}^{\F_k/D_k \cup D_{-k}} =N_{g,\rel}^{\mu\nu\emptyset\emptyset} \,,
\end{equation}
where $N_{g,\rel}^{\mu\nu\emptyset\emptyset}$ is the specialization for $\rho=\sigma=\emptyset$ of the invariants $N_{g,\rel}^{\mu\nu\rho\sigma}$ defined in \eqref{eq_gw_blowup}.
\end{itemize}
\end{lem}

\begin{proof}
By \eqref{eq_vdim}, the virtual dimension of the moduli stack of relative stable maps used to define 
\begin{equation}\label{eq_gw_2} \langle \mu,\delta^2|\, (-1)^g \lambda_g;\alpha_1|\,\nu,\delta^1 \rangle_{g,\beta}^{\F_k/D_k \cup D_{-k}}    
\end{equation}
is 
\begin{align}\label{eq_vdim_2}
&g-1+\beta \cdot (D_k+D_{-k}+2F)+1-(\beta \cdot D_{-k} -\ell(\mu))-(\beta \cdot D_k -\ell(\nu)) \\
&=g+2h+\ell(\mu)+\ell(\nu)\,.\nonumber
\end{align}
On the other hand, we integrate 
in \ref{eq_gw_1}
over the virtual dimension class a cohomology class of complex degree 
\begin{equation} \label{eq_insert_2}
j+2+\ell(\mu)+\ell(\nu)-s\,.
\end{equation}
\eqref{eq_gw_2} is $0$ unless 
\eqref{eq_vdim_2} = \eqref{eq_insert_2}, that is 
$2h+s+(g-j)=2$. As $h$, $s$, and $(g-j)$ are nonnegative integers, there are only four possibilities: $(h,s,j)=(0,0,g-2)$, 
$(h,s,j)=(0,1,g-1)$,
$(h,s,j)=(0,2,g)$ and $(h,s,j)=(1,0,g)$.

If $(h,s,j)=(0,2-r,g-r)$ for some $r \in\{0,1,2\}$, then $\beta=d F$ and we fix the position of $\ell(\mu)+\ell(\nu)+r-2$
of the $\ell(\mu)+\ell(\nu)$ contact points with 
$D_k \cup D_{-k}$.
If $\ell(\mu)+\ell(\nu)+r-2\geq 1$, we can choose the positions of one of the contact point and of the interior marked point in two different fibers of $p \colon \F_k \rightarrow \PP^1$, and so, as 
a curve of class $\beta=dF$ is contained in a $\PP^1$-fiber of $p$, the set of curves matching the constraints is empty. Hence,  \eqref{eq_gw_2} is still zero unless $r=0$ and $\ell(\mu)=\ell(\nu)=1$. Thus, we are in the case (i) and \eqref{eq_lem_pre_2} follows from 
Lemma \ref{lem_vanishing_2}(i).

If $(h,s,j)=(1,0,g)$, then we are in the case (ii) and 
\eqref{eq_lem_pre_3} follows from the definition 
\eqref{eq_gw_blowup} of $N_{g,\rel}^{\mu\nu\rho\sigma}$.
\end{proof}

\subsection{Gromov-Witten invariants of $h$-transverse toric surfaces} \label{section_relative_gw}

Let $\Delta$ be a $h$-transverse balanced collection of vectors in $\Z^2$ as in Definition 
\ref{def_h_transverse}, and $n$ a nonnegative integer.
We assume that the corresponding toric surface $X_\Delta$ given by Definition \ref{def_X_delta} is smooth. We defined in \S \ref{section_h_transverse} a curve class $\beta_{\Delta}$ and two smooth disjoint divisors 
$D_b$ and $D_t$. 
We define the relative Gromov-Witten
invariants $N_{g,\rel}^{\Delta,n}$ of 
$X_\Delta/D_b \cup D_t$ by 
\begin{equation} \label{eq_gw_h}
N_{g,\rel}^{\Delta,n} \coloneqq \langle \mu,\delta^2|(-1)^{g-g_\Delta} \lambda_{g-g_\Delta}; \alpha_1,\dots,\alpha_n|\nu,\delta^1\rangle_{g,\beta_\Delta}^{X_\Delta/D_b\cup D_t} \,,
\end{equation}
where we apply the general definition \eqref{eq_gw} of relative Gromov-Witten invariants to $X=X_\Delta$, $D_1=D_b$, $D_2=D_t$, $\beta=\beta_\Delta$, and where:
\begin{itemize}
\item[(i)] $\mu$ is the partition of $d_t=\beta \cdot D_t$ whose all parts are  equal $1$, that is $\ell(\mu)=d_t$, and $\delta^2=(\delta^2_j)_{1\leq j\leq d_t}$, where $\delta^2_j=1 \in H^0(D_t,\Z)$ for all $j$,
\item[(ii)] $\nu$ is the partition of 
$d_b=\beta \cdot D_b$ whose all all parts are equal to $1$, that is $\ell(\nu)=d_b$, and 
$\delta^1=(\delta^1_j)_{1\leq j\leq d_b}$, where 
$\delta_j^1=1 \in H^0(D_b,\Z)$ for all $j$,
\item[(iii)] the cohomology classes 
$\alpha_1,\dots,\alpha_n$ inserted at the $n$ interior marked points are all equal to the Poincaré dual class of a point in $H^4(X_\Delta,\Z)$
\item[(iv)] the class $\gamma$ in \eqref{eq_gw} on the moduli stack of relative stable maps is taken to be $(-1)^{g-g_{\Delta,n}}\lambda_{g-g_{\Delta,n}}$, where $g_{\Delta,n}=n+1-|\Delta|$
(see Definition \ref{def_floor_diagram}(ii))
and $\lambda_{g-g_{\Delta,n}}$ the lambda class of complex degree $g-g_{\Delta,n}$ as in \eqref{eq_lambda}.
\end{itemize}

In other words,  $N_{g,\rel}^{\Delta,n}$ is a virtual count of genus $g$ class $\beta_\Delta$ curves
in $X_\Delta$, with transversal intersections with the divisors $D_b$ and $D_t$ and passing through 
$n$ given points in $X_\Delta$.

\subsection{The degeneration formula for the invariants $N_{g,\rel}^{\Delta,n}$}
\label{section_degeneration}

In this section, we fix $\Delta=\Delta_{h,d}^{\F_k}$, as in \S \ref{ex_hirzebruch},
so that $X_\Delta=\F_k$ and we state in Lemma \ref{lem_degeneration} a precise version of the degeneration formula \cite{MR1938113} computing the relative Gromov-Witten invariants $N_{g,\rel}^{\Delta,n}$ defined in \eqref{eq_gw_h}.

\begin{figure}
\center{\scalebox{.3}{\input{F2.pspdftex}}}
\caption{The degeneration $\mathcal{X} \rightarrow \A^1$ of $\F_k$
to a chain of $n$ copies of $\F_k$, and the degeneration of the $n$ point conditions.}
\label{Fig:degeneration}
\end{figure}

By successive degeneration of $\F_k$ to the normal cone of the divisor $D_{-k}$, we construct a degeneration $\cX \rightarrow \A^1$ of 
$\F_k$, whose special fiber $\cX_0$ is a chain of $n$ copies $\F_k^{(j)}$ of $\F_k$
$1 \leq j \leq n$. For every 
$1 \leq j \leq n-1$, the divisor 
$D_{-k}^{(j)}$ of $\F_k^{(j)}$ is transversally glued with the divisor $D_k^{(j+1)}$ of $\F_k^{(j+1)}$. 
We denote 
\begin{equation}
D^0 \coloneqq D_{k}^{(0)}\,,\,\,\,\, 
D^j \coloneqq D_{-k}^{(j)}\simeq D_k^{j+1}\,
\text{for} \,\,1 \leq j\leq n-1, \,\,\,\,\text{and}\,\,\, D^{n} \coloneqq D_{-k}^{(n)}\,.
\end{equation} 
Our goal is to state precisely the degeneration formula in relative Gromov-Witten theory \cite{MR1938113} applied to 
$\cX \rightarrow \A^1$
to compute $N_{g,\rel}^{\Delta,n}$, where
we distribute the $n$ point conditions $\alpha_1,\dots,\alpha_n$ appearing in
\eqref{eq_gw_h} by placing one on each of the 
$n$ components $\F_k^{(1)},\dots,\F_k^{(n)}$
(see Figure \ref{Fig:degeneration}). To do that, we need to introduce some combinatorial notations, that might seem heavy but are geometrically completely natural. We first introduce in Definition \ref{def_G} below a set of weighted decorated graphs which will index the terms of the degeneration formula and describe the possible degeneration types of curves in the special fiber 
$\cX_0$. The geometric meaning of each condition is explained in the proof of Lemma \ref{lem_degeneration} 
below.

\begin{defn} \label{def_G}
We denote by $\cG^{k,n}_{h,d,g}$ the set of decorated weighted graphs $\Gamma$ which are as follows.
\begin{itemize}
\item[(i)] $\Gamma$ is a weighted graph as in 
\S \ref{section_floor}, with a set $V(\Gamma)$ of vertices and a set $E(\Gamma)$ of edges,
which are either bounded or unbounded. Every edge $E \in E(\Gamma)$ has a weight $w_E \in \Z_{\geq 1}$, which is equal to $1$ if $E$ is unbounded.
\item[(ii)] Every vertex $V \in V(\Gamma)$ has a genus decoration $g_V \in \Z_{\geq 0}$, and the first Betti number $g_\Gamma$ of $\Gamma$ is such that 
$g_\Gamma + \sum_{V \in V(\Gamma)} g_V=g$.
\item[(iii)] Every vertex $V \in V(\Gamma)$
is decorated by an index $j_V \in \{1,\dots,n\}$ and a class $\beta_V =h_V D_k^{(j_V)}+d_V F^{(j_V)} \in H_2(\F_k^{(j_V)},\Z)$.
\item[(iv)] For every $1 \leq j \leq n$, there is a distinguished vertex, denoted by $V_j$, among the vertices with $j_V=j$.
\item[(v)] Every edge $E \in E(\Gamma)$ is decorated by an index $j_E\in \{0,\dots,n\}$. 
If $E$ is a bounded edge, then $j_E \in \{1,\dots,n-1\}$ and there is a labeling $V$, $V'$ of the two vertices incident to $E$ such that 
$j_V=j_E$  and $j_{V'}=j_E+1$. If $E$ is an unbounded edge, then $j_E\in \{0,n\}$, and 
there are exactly $d+kh$ unbounded edges $E$ with 
$j_E=0$ and $d$ unbounded edges $E$ with $j_E=n$.
For every $V \in V(\Gamma)$, there are exactly 
$d_V$ edges $E$ incident to $V$ with 
$j_E=j_V$ and $d_V+kh_V$ egdes $E$
incident to $V$ with $j_E=j_V-1$.
\item[(vi)] Every half-edge, that is a pair
$(V,E)$ with $V \in V(\Gamma)$ and $E \in E(\Gamma)$ incident to $V$, is decorated by a cohomology class $c_{(V,E)} \in \{1,p_{j_E}\}$,
where $1 \in H^0(D^{j_E},\Z)$ and $p_{j_E} \in H^2(D^{j_E},\Z)$ is Poincaré dual to a point.   
If $E$ is unbounded, with incident vertex $V$, then $c_{(V,E)}=1$. If $E$ is bounded, then for exactly one vertex $V'$ incident to $E$ we have 
$c_{(V',E)}=1$, and for the other vertex $V''$ incident to $E$, we have $c_{(V'',E)}=p_{j_E}$.
\item[(vii)] Every vertex $V \in V(\Gamma)$ is decorated by an index $0 \leq m_V \leq g_V$, and we have 
$\sum_{V \in V(\Gamma)} m_V=g$.
\end{itemize}
\end{defn}

\begin{figure}
\center{\scalebox{.3}{\input{F3.pspdftex}}}
\caption{On the left, a curve in the special fiber of the degeneration. In the middle, the corresponding dual graph $\Gamma$. On the right, the corresponding marked floor diagram $\fD_\Gamma$.}
\label{Fig:graphs}
\end{figure}

\begin{defn}
For every $\Gamma \in \cG^{k,n}_{h,d,g}$ and 
$V \in V(\Gamma)$, let $E_{V,-}$ (resp.\ $E_{V,+}$) be the set of $E \in E(\Gamma)$ incident to $V$ such that $j_E=j_V$ (resp.\ $j_E=j_V-1$). We denote 
\begin{equation} \mu_V \coloneqq (w_E)_{E \in E_{V,-}}\,, \,\,\, \nu_V \coloneqq (w_E)_{E \in E_{V,+}} 
\end{equation}
and
\begin{equation}
\delta^2_V \coloneqq (c_{(V,E)})_{E \in E_{V,-}}
\,, \,\,\,
\delta^1_V \coloneqq (c_{(V,E)})_{E \in E_{V,+}}\,.
\end{equation}
We define a relative Gromov-Witten invariant of 
$\F_k^{(j_V)}$ relatively 
to $D^{j_V-1} \cup D^{j_V}$ by
\begin{equation} \label{eq_N_V}
N_{\Gamma, V} \coloneqq 
\begin{cases}
\langle \mu_V, \delta^2_V| 
(-1)^{m_V} \lambda_{m_V}| \nu_V,\delta^1_V \rangle_{g_V,\beta_V}^{\F_k/D^{j_V-1} \cup D^{j_V}} \,\textit{if} \,\,\, V \neq V_{j_V}\\
\langle \mu_V, \delta^2_V| 
(-1)^{m_V} \lambda_{m_V};\alpha_{j_V}| \nu_V,\delta^1_V \rangle_{g_V,\beta_V}^{\F_k/D^{j_V-1} \cup D^{j_V}} \,\textit{if} \,\,\, V=V_{j_V}\,.
\end{cases}
\end{equation}
\end{defn}

\begin{lem}\label{lem_degeneration}
For $\Delta=\Delta_{h,d}^{\F_k}$,
the relative Gromov--Witten invariants 
$N_{g,\rel}^{\Delta,n}$ of $X_{\Delta}=\F_k$ defined in 
\eqref{eq_gw_h} are given by:
\begin{equation} \label{eq_degeneration}
N_{g,\rel}^{\Delta_{h,d}^{\F_k},n} 
= \sum_{\Gamma \in \cG_{h,d,g}^{k,n}}
\frac{\prod_{E \in E(\Gamma)} w_E}{|\Aut(\Gamma)|}
\prod_{V \in V(\Gamma)} N_{\Gamma, V}\,,
\end{equation}
where $|\Aut(\Gamma)|$ is the order of the group of permutation symmetries of $\Gamma$ as decorated weighted graph.
\end{lem}

\begin{proof}
We claim that \eqref{eq_degeneration} is the degeneration formula in relative Gromov-Witten theory \cite{MR1938113} applied to 
$\cX \rightarrow \A^1$
to compute $N_{g,\rel}^{\Delta_{h,d}^{\F_k},n}$, where
we distribute the $n$ point conditions $\alpha_1,\dots,\alpha_n$ appearing in
\eqref{eq_gw_h} by placing one on each of the 
$n$ components $\F_k^{(1)},\dots,\F_k^{(n)}$.
Indeed, a graph $\Gamma \in \cG_{h,d,g}^{k,n}$ as in Definition \ref{def_G} indexes a moduli space of stable maps to the special fiber $\cX_0$ that have virtually generically dual graph $\Gamma$: 
a vertex $V$ corresponds to a curve 
of genus $g_V$ and class $\beta_V$ contained in the component $\F_k^{(j_V)}$, the vertex $V_j$ corresponds to the curve with the $j$-th interior marked point where the point condition $\alpha_j$ is imposed, a bounded edge $E$ corresponds to a node on the divisor $D^{j_E}$, and an unbounded edge $E$ corresponds to a relative marking on either $D^0$ or $D^n$ (see Figure \ref{Fig:graphs}).  
Furthermore, the cohomological decorations $c_{(V,E)}$ implement the insertion of the diagonal class in the degeneration formula of \cite{MR1938113}, where we used the fact the the class of the diagonal $D^j \simeq \PP^1 \subset D^j \times D^j \simeq \PP^1 \times \PP^1$ is $1 \times p_j +p_j \times 1$. Finally, in the definition \ref{eq_gw_h} of 
$N_{g,\rel}^{\Delta_{h,d}^{\F_k},n}$, there is an insertion of the class $(-1)^g \lambda_g$ and we used \eqref{eq_lambda_gluing} to split the lambda class: the index $m_V$ in Definition \ref{def_G}(vii) means that we insert the class $(-1)^{m_V}
\lambda_{m_V}$ on the vertex $V$, and it is indeed what we did in the definition \eqref{eq_N_V} of $N_{\Gamma, V}$. 
\end{proof}

We use the following terminology in \S \ref{subsection_main_result} below. 

\begin{defn}\label{def_vertices_type}
Given a decorated weighted graph 
$\Gamma \in \cG_{h,d,g}^{k,n}$, and a vertex $V \in V(\Gamma)$, we say that:
\begin{itemize}
\item[(i)]$V$ is of type $A$ if $V \neq V_{j_V}$,
$h_V=0$, $\mu_V$ and $\nu_V$ are both the trivial $1$-part partition of $d_V$
(in particular $V$ is bivalent), one of the edges incident to $V$ has $c_{(V,E)}=1$ and the other has $c_{(V,E')}=p_{j_{E'}}$, and $m_V=g_V=0$.
\item[(ii)]$V$ is of type $B$ if $V=V_{j_V}$,
$h_V=0$, $\mu_V$ and $\nu_V$ are both the trivial $1$-part partition of $d_V$,
(in particular $V$ is bivalent), both edges incident to $V$ have $c_{(V,E)}=1$, and $m_V=g_V=0$.
\item[(iii)]$V$ is of type $C$ if $V=V_{j_V}$, 
$h_V=1$, all edges incident to $V$ have $c_{(V,E)}
=p_{j_E}$, and $m_V=g_V$. 
\end{itemize}
We denote by $\tilde{\cG}_{h,d,g}^{k,n}$ the set of graphs $\Gamma \in \cG_{h,d,g}^{k,n}$
whose vertices are all of type $A$, $B$ or $C$. 
\end{defn}

\begin{figure}
\center{\scalebox{.3}{\input{2.pspdftex}}}
\caption{A chain of edges in $\Gamma$ connected by bivalent vertices of type $A$ or $B$, and with endpoints of type $C$.}
\label{Fig:chain}
\end{figure}

\begin{lem} \label{lem_vertex_B}
Let $\Gamma \in \tilde{\cG}_{h,d,g}^{k,n}$. 
Let $c$ be a chain of edges in $\Gamma$ connected by bivalent vertices of type $A$ or $B$. Assume that the endpoints of $c$ are vertices of type $C$.
Then, the chain $c$ contains exactly one vertex of type $B$.
\end{lem}

\begin{proof}
We consider a chain $c$ with edges $E^{(1)},\dots,E^{(m)}$, connected
by bivalent vertices $V^{(1)},\dots, V^{(m-1)}$ of type $A$ and $B$. Assume first that the chain $c$ has two endpoints: we denote by $V^{0}$ and $V^{(m)}$ the vertices of type $C$ incident to 
$E^{(0)}$ and $E^{(m)}$
(see Figure \ref{Fig:chain}). 
We say that the edge $E^{(\ell)}$ is of type $[1,p]$ (resp.\ $[p,1]$)
if 
$c_{(V^{(\ell-1)},E^{(\ell)})}=1$ and
$c_{(V^{(\ell)},E^{(\ell)})}=p_{j_{E^{(\ell)}}}$
(resp.\ $c_{(V^{(\ell-1)},E^{(\ell)})}=p_{j_{E^{(\ell)}}}$ and
$c_{(V^{(\ell)},E^{(\ell)})}=1$).
By Definition \ref{def_G}, every edge is either of type $[1,p]$ or of type $[p,1]$. 

By Definition \ref{def_vertices_type}(i), the type of edges ``propagates" through vertices of type $A$:
if $V^{(\ell)}$ is of type $A$ and $E^{(\ell)}$ is of type $[p,1]$ (resp.\ $[1,p]$), then $E^{(\ell+1)}$ is of type $[p,1]$ (resp.\ $[1,p])$, whereas
by Definition \ref{def_vertices_type}(ii) a vertex of type $B$ ``flips" an edge $[p,1]$ into an edge of type $[1,p]$: if 
$V^{(\ell)}$ is of type $B$, then $E^{(\ell)}$ is of type $[p,1]$ and $E^{(\ell+1)}$ is of type 
$[1,p]$.
On the other hand, by Definition \ref{def_vertices_type}(iii) of vertices of type $C$, $E^{(1)}$ is of type $[p,1]$ and that 
$E^{(m)}$ is of type $[1,p]$, so the type of the edges needs to flip at some vertex from $[p,1]$ to $[1,p]$ and this vertex is of type $B$. It is the unique vertex of type $B$ as there is no type of vertex able to flip back $[1,p]$ to $[p,1]$.

If the chain $c$ has one or zero endpoints, that is if $E^{(0)}$ or $E^{(m)}$ are unbounded, the same argument applies using that $c_{(V,E)}=1$ if $E$ is unbounded by Definition \ref{def_G}(vi). 
\end{proof}

In \S \ref{subsection_main_result} below, we use the following construction of a marked 
$(\Delta_{h,d}^{\F_k},n)$-floor diagram
$\fD_\Gamma$ starting from a decorated weighted graph $\Gamma \in \tilde{\cG}_{h,d,g}^{k,n}$. 

\begin{defn} \label{def_D_Gamma}
Let $\Gamma \in \tilde{\cG}_{h,d,g}^{k,n}$. We define a marked oriented weighted graph $\fD_\Gamma$ as follows (see Figure \ref{Fig:graphs}).
The vertices of $\fD_\Gamma$ are the vertices of $\Gamma$ of type $C$, and each such vertex $V$ is marked by the index $j_V \in \{1,\dots,n\}$. 
Moreover, for each chain $c$ of edges in $\Gamma$ connected by bivalent vertices of type $A$ or $B$
and with endpoints of type $C$, there is an edge 
$E$ in $\fD_\Gamma$ incident to the endpoints of $c$. We define the weight $w_E$ of $E$ as the common weight in $\Gamma$ of the edges contained in $c$. The edge $E$ is marked by $j_{V} \in \{1,\dots,n\}$, where $V$
is the unique vertex of type $B$ contained in $c$ given by Lemma \ref{lem_vertex_B}. Finally, we orient the edges of $\fD_\Gamma$ so that the marking is increasing.
\end{defn}

\begin{lem} \label{lem_floor_diagram}
For every $\Gamma \in \tilde{\cG}_{h,d,g}^{k,n}$, 
the marked oriented weighted graph $\fD_\Gamma$ 
defined in Definition \ref{def_D_Gamma} is a marked $(\Delta_{h,d}^{\F_k},n)$-floor diagram as in Definitions 
\ref{def_floor_diagram}-\ref{def_marking}.
\end{lem}

\begin{proof}
It is mostly a direct consequence of Definitions \ref{def_G}-\ref{def_vertices_type}-\ref{def_D_Gamma}. The only part which requires an argument is why $\fD_\Gamma$ has first Betti number $g_{\Delta,n}=n+1-|\Delta|$. Let 
$V(\fD_\Gamma)$ (resp.\ $E^b(\fD_\Gamma)$,
$E^\infty(\fD_\Gamma)$) be the set of vertices 
(resp.\ bounded and unbounded edges) of $\fD_\Gamma$. By construction, we have 
$|E^b(\fD_\Gamma)|+|E^\infty(\fD_\Gamma)|+|V(\fD_\Gamma)|=n$ and $2|V(\fD_\Gamma)|+|E^{\infty}(\fD_\Gamma|=|\Delta|$, and so the first Betti number of $\fD_\Gamma$ is 
$|E^b(\fD_\Gamma)|-|V(\fD_\Gamma)|+1=n+1-|\Delta|$.
\end{proof}

\begin{lem}\label{lem_mult_identity}
For every $\Gamma \in \tilde{\cG}_{h,d,g}^{k,n}$, 
denoting by $E(\Gamma)$ the set of edges of 
$\Gamma$ and by $V^A(\Gamma)$ the set of vertices of $\Gamma$ of type $A$ , we have 
\begin{equation} \label{eq_mult_identity}
\prod_{E \in E(\Gamma)} w_E \prod_{V \in V^A(\Gamma)}\frac{1}{d_V} = \prod_{E \in E(\fD_\Gamma)}w_E^2\,,
\end{equation}
where $E(\fD_\Gamma)$ is the set of edges of 
$\fD_\Gamma$.
\end{lem}

\begin{proof}
Let $E^{(1)},\dots, E^{(m)}$ be a chain of edges of $\Gamma$ connected by bivalent vertices of type $A$ or $B$ and with endpoints of type $C$. 
All the edges of the chain have the same weight $w$
and all the vertices of type $A$ or $B$ have 
$d_V=w$. By Lemma \ref{lem_vertex_B}, the chain contains one vertex of type $B$ and $m-2$ vertices of type $A$ and so the contribution of the chain to the left-hand side of \eqref{eq_mult_identity} is $\frac{w^m}{w^{m-2}}=w^2$. 
\end{proof}

\subsection{Main result} \label{subsection_main_result}
In Lemma \ref{prop_deg_formula} below, we express for $\Delta=\Delta_d^{\PP^2}$ and 
$\Delta = \Delta_{h,d}^{\F_k}$
the relative Gromov-Witten invariants 
$N_{g,\rel}^{\Delta,n}$ defined in \eqref{eq_gw_h}
in terms of floor diagrams whose vertices are weighted by the relative Gromov-Witten invariants 
$N_{g,\rel}^{\mu\nu\emptyset\emptyset}$ of Hirzebruch surfaces defined in \eqref{eq_gw_blowup}.

\begin{lem} \label{prop_deg_formula}
Let $\Delta$ be a $h$-transverse balanced collection of vectors in $\Z^2$ of the form
$\Delta_d^{\PP^2}$ or $\Delta_{h,d}^{\F_k}$
as in Examples \ref{ex_P2}-\ref{ex_hirzebruch}.
Let $n$ be a nonnegative integer such that $g_{\Delta,n} \geq 0$.
Then the relative Gromov--Witten invariants 
$N_{g,\rel}^{\Delta,n}$ of $X_\Delta$ defined in 
\eqref{eq_gw_h} are given by:
\begin{equation} \label{eq_prop_deg_formula}\sum_{g \geq g_{\Delta,n}}
N_{g,\rel}^{\Delta,n} u^{2g-2+d_b+d_t}
=\end{equation}
\[\sum_{\fD}
\left(
\prod_{E \in E(\fD)} w_E^2 \right)
\left( \prod_{V \in V(\fD)} 
\sum_{g \geq 0} N_{g,\rel}^{\mu_V \nu_V \emptyset\emptyset} u^{2g-2+\ell(\mu_V)+\ell(\nu_V)}
\right)\,, \]
where the sum over $\fD$ is over the isomorphism classes of marked 
$(\Delta,n)$-floor diagrams $\fD$ as
in Definition \ref{def_floor_diagram}, 
$E(\fD)$ is the set of edges of $\fD$, $w_E$ is the weight of the edge 
$E \in E(\fD)$, $V(\fD)$ is the set of vertices of
$\fD$, and for every vertex $V \in V(\fD)$, 
$\mu_V$ (resp.\ $\nu_V$) is the partition  
whose parts are the weights of outgoing
(resp.\ ingoing) edges of $\fD$
incident to $V$. Moreover, $N_{g,\rel}^{\mu_V \nu_V\emptyset\emptyset}$ is the specialization 
$\mu=\mu_V$, $\nu=\nu_V$, $\rho=\sigma=\emptyset$ of the relative Gromov-Witten invariants 
$N_{g,\rel}^{\mu\nu\rho\sigma}$ defined in 
\eqref{eq_gw_blowup}.
\end{lem}

\begin{proof}
We first assume that $\Delta = \Delta_{h,d}^{\F_k}$, and so in particular $X_{\Delta}=\F_k$.
By Lemma \ref{lem_degeneration}, 
$N_{g,\rel}^{\Delta,n}$ is expressed by 
\eqref{eq_degeneration} in terms of the invariants 
$N_{\Gamma,V}$ defined in \eqref{eq_N_V}.
Recall tha we introduced in Definition \ref{def_vertices_type} the notions of vertices of type $A$, $B$ or $C$. By Lemma \ref{lem_dim_1}, 
if $V \neq V_{j_V}$, then $N_{\Gamma,V}=0$ unless $V$ is of type $A$, and $N_{\Gamma,V}=\frac{1}{d_V}$ if $V$ is of type $A$. By Lemma \ref{lem_dim_2}, if $V =V_{j_V}$, then 
$N_{\Gamma,V}=0$ unless $V$ is of type $B$ or $C$, 
$N_{\Gamma,V}=1$ if $V$ is of type $B$, and 
$N_{\Gamma,V}=N_{g_V,\rel}^{\mu_V\nu_V \emptyset \emptyset}$ if $V$ is of type $C$. 
Hence, we can rewrite 
\eqref{eq_degeneration} as a sum over the set 
$\tilde{\cG}_{h,d,g}^{k,n}$ of graphs $\Gamma \in \cG_{h,d,g}^{k,n}$ whose vertices are all of type $A$, $B$ or $C$: 
\begin{equation} 
N_{g,\rel}^{\Delta,n} 
= \sum_{\Gamma \in \tilde{\cG}_{h,d,g}^{k,n}}
\frac{\prod_{E \in E(\Gamma)} w_E}{|\Aut(\Gamma)|}
\prod_{V \in V^A(\Gamma)} \frac{1}{d_V} 
\prod_{V \in V^C(\Gamma)}N_{g_V, \rel}^{\mu_V\nu_V \emptyset \emptyset}\,,
\end{equation}
where $V^A(\Gamma)$ (resp.\ $V^B(\Gamma)$, $V^C(\Gamma)$) is the set of vertices of 
$\Gamma$ of type $A$ (resp.\ $B$, $C$).
In Definition \ref{def_D_Gamma}-Lemma \ref{lem_floor_diagram}, we defined a marked 
$(\Delta,n)$ floor diagram $\fD_\Gamma$ for every $\Gamma \in \tilde{\cG}_{h,d,g}^{k,n}$, and every marked $(\Delta,n)$-floor diagram can be uniquely obtained that way.  
Thus,
\eqref{eq_prop_deg_formula}
follows from Lemma \ref{lem_mult_identity}.

For $\Delta=\Delta_d^{\PP^2}$, we follow exactly the same sequence of aguments by adapting the results of \S \ref{section_degeneration} to the degeneration of $\PP^2$ to a chain made of one copy of $\PP^2$ and $n$ copies of $\F_1$, where each of the $n$ point insertions is inserted in a copy of $\F_1$.
\end{proof}

The following theorem is the main result of the present paper. It is a precise version of Theorem \ref{main_thm0} stated in the introduction.

\begin{thm} \label{main_thm_precise}
Let $\Delta$ be a $h$-transverse balanced collection of vectors in 
$\Z^2$ of the form $\Delta^{\PP^2}_d$ or $\Delta^{\F_k}_{h,d}$, and let 
$n$ be a nonnegative integer such that $g_{\Delta,n} \geq 0$. Then we have 
the equality 
\begin{equation} \label{eq_main}
\sum_{g \geq g_{\Delta,n}} N_{g,\rel}^{\Delta,n} u^{2g-2+d_b+d_t}=
u^{d_b+d_t-|\Delta|}
N^{\Delta,n}_{\floor}(q^{\frac{1}{2}})
\left((-i)
(q^{\frac{1}{2}} - q^{-\frac{1}{2}})
\right)^{2 g_{\Delta,n}-2+|\Delta|} 
\end{equation}
of power series in $u$ with rational coefficients,
where 
\begin{equation} q=e^{iu}=\sum_{m \geq 0} \frac{(iu)^m}{m!} \,.
\end{equation}
\end{thm}

\begin{proof}
Lemma \ref{prop_deg_formula} expresses $N_{g,\rel}^{\Delta,n}$ in terms of the invariants 
$N_{g,\rel}^{\mu\nu\emptyset\emptyset}$, which are computed in Theorem \ref{key_thm}. We obtain
\begin{align} 
&\sum_{g \geq g_{\Delta,n}}
N_{g,\rel}^{\Delta,n} u^{2g-2+d_b+d_t}= \\
&
u^{d_b+d_t-|\Delta|}\left((-i)
(q^{\frac{1}{2}} - q^{-\frac{1}{2}})
\right)^{2 g_{\Delta,n}-2+|\Delta|} 
\sum_{\fD} \prod_{E \in E(\fD)}w_E^2 \prod_{V \in V(\fD)}
\prod_{j=1}^{\ell(\mu_V)}\frac{[\mu_{V,j}]_q}{\mu_{V,j}} 
\prod_{l=1}^{\ell(\nu_V)}
\frac{[\nu_{V,l}]_q}{\nu_{V,l}} \,,
\end{align}
where we used that $\sum_{V}(\ell(\mu_V)+\ell(\nu_V))=2g_{\Delta,n}-2+|\Delta|$.
As every edge $E$ of 
$\fD$ with $w(E) \neq 1$ is connected to two vertices, we have
\begin{equation}
\prod_{E \in E(\fD)}w_E^2 
\prod_{V \in V(\fD)}
\prod_{j=1}^{\ell(\mu_V)}\frac{[\mu_{V,j}]_q}{\mu_{V,j}} 
\prod_{l=1}^{\ell(\nu_V)}
\frac{[\nu_{V,l}]_q}{\nu_{V,l}} 
= \prod_{E \in E(\fD)} [w_E]_q^2 =
m_{\fD}(q^{\frac{1}{2}}) \,,
\end{equation}
where we used the Definition \ref{def_q_mult} of the $q$-refined multiplicity of a floor diagram 
$\fD$. We conclude using the Definition \ref{def_N_delta_q} of $N^{\Delta,n}_{\floor}(q^{\frac{1}{2}})$. 
\end{proof}

The relation between relative Gromov-Witten invariants and $q$-refined floor diagrams 
given by Theorem \ref{main_thm_precise} can be read and exploited in both directions. 
First, the relative Gromov--Witten invariants $N_{g,\rel}^{\Delta,n}$ give an algebro-geometric realization of the
$q$-refined counts of floor diagrams $N^{\Delta,n}_{\floor}(q^{\frac{1}{2}})$ which is independent of tropical geometry. An illustration of the use of this geometric point of view to solve a priori purely combinatorial questions about the $q$-refined counts of floor diagrams is given by our proof of the $q$-refined Abramovich-Bertram formula in \S \ref{section_brugalle_conj}.
Conversely, Theorem \ref{main_thm_precise} can be viewed as providing a convenient tool to compute the relative Gromov-Witten invariants $N_{g,\rel}^{\Delta,n}$. Indeed, the combinatorial enumeration of floor diagrams allows for efficient calculations, as shown for example in 
\cite{MR2500574, arroyo2011recursive, MR3433282, MR3345189, MR3658147}.

Theorem \ref{main_thm_precise}
is analogous to the main result of 
\cite{MR3904449} which relates Block-Göttsche $q$-refined tropical curve counts and higher genus log Gromov-Witten invariants of toric surfaces with a lambda class insertion. Combining these two results we obtain in \S \ref{section_log} 
 a non-trivial comparison result (Theorem \ref{thm_log}) between log and relative Gromov-Witten invariants for $\PP^2$ and Hirzebruch surfaces.

\section{Dimension $3$ and stable pairs}
\label{section_dim_3}
In \S \ref{section_gw_3d}, we express the relative Gromov-Witten invariants 
$N_{g,\rel}^{\Delta,n}$ of a $h$-transverse toric surface $X_\Delta$ in terms of equivariant 
relative Gromov-Witten invariants of the $3$-fold $X_\Delta \times \A^1$. In \S \ref{section_pt}, we reformulate our main result, Theorem \ref{main_thm_precise} in terms of Pandharipande-Thomas stable pair invariants of $X_\Delta \times \A^1$.

\subsection{Relative Gromov-Witten invariants of $X_\Delta \times \A^1$}
\label{section_gw_3d}
In \eqref{eq_gw_h}, we defined the relative Gromov-Witten invariants 
$N_{g,\rel}^{\Delta,n}$ of a $h$-transverse toric surface $X_\Delta$
relatively to the divisor $D_b \cup D_t$. The definition \eqref{eq_gw_h}
involves the insertion of a lambda class 
$(-1)^{g-g_{\Delta,n}} \lambda_{g-g_{\Delta,n}}$. Following \cite[\S 1]{mpt}, we can interpret this lambda class insertion in Gromov-Witten theory of the surface $X_\Delta$ as an excess class coming from Gromov-Witten theory of the $3$-fold 
$X_\Delta \times \A^1$. Moduli spaces of stable maps to 
$X_\Delta \times \A^1$ are non-compact but admit a $\C^{*}$-action coming from the 
$\C^{*}$-action scaling the second factor in $X_\Delta \times \A^1$.
The $\C^{*}$-fixed locus being compact, one can define equivariant Gromov-Witten invariants. We define the equivariant relative Gromov-Witten invariants of 
$(X_\Delta \times \A^1)/((D_b \cup D_t)\times \A^1)$:
\begin{equation} \label{eq_gw_3d}
N_{g,\rel}^{\Delta,n,3d} \coloneqq \langle \mu,\delta^2| \alpha_1,\dots,\alpha_n|\nu,\delta^1\rangle_{g,\beta_\Delta}^{X_\Delta \times \A^1/(D_b\cup D_t)\times \A^1} \,,
\end{equation}
where all the cohomology classes on $X_\Delta$ in \eqref{eq_gw_h}
are interpreted as their pullback on $X_\Delta \times \A^1$ by the projection on the first factor. The equivariant invariant $N_{g,\rel}^{\Delta,n,3d}$ is an element of $\Q(t)$, where $t$ is the equivariant parameter. 
Following \cite[\S 1]{mpt}, we can apply the localization formula to obtain:
\begin{equation}\label{eq:2d_3d}
    N_{g,\rel}^{\Delta,n,3d}= t^{g_{\Delta,n}-1} N_{g,\rel}^{\Delta,n}\,.
\end{equation}
The right-hand side of \eqref{eq:2d_3d} is a $3$-fold invariant defined without lambda class insertion, whereas the left-hand side of \eqref{eq:2d_3d} is a
surface invariant defined with lambda class insertion.

\subsection{Stable pair invariants of $X_\Delta \times \A^1$}
\label{section_pt}
Higher genus Gromov-Witten theory of 3-folds is conjecturally equivalent to the sheaf counting theories given by Donaldson-Thomas counts of ideal sheaves \cite{mnop1, mnop2}
and 
Pandharipande-Thomas counts of stable pairs \cite{pt}. 

We consider the moduli space of relative stable pairs
 $\cO_X \xrightarrow{s} F$ on $X_\Delta \times \A^1 /((D_b \cup D_t) \times \A^1)$, where the sheaf $F$ has support of class $\beta_\Delta$ and Euler characteristic $\chi(F)=m$.
We denote by $P_{m,\rel}^{\Delta,n}$ the equivariant
relative stable pair invariant of $X_\Delta \times \A^1 /((D_b \cup D_t) \times \A^1)$ extracted from this moduli space by inserting $n$ times the class pullback of the class Poincaré dual to a point in $X_\Delta$, by considering 
contact orders with $D_b \cup D_t$ 
defined by partitions whose parts are all $1$, and 
by inserting  the class $1$ at each contact point with $D_b \cup D_t$.
We refer to \cite{mnop2, pt, liwu} for the theory of relative stable pairs and for the definition of insertions. The equivariant invariant $P_{m,\rel}^{\Delta,n}$ is an element of $\Q(t)$, where $t$ is the equivariant parameter. Finally, we define
\begin{equation}
    P_\rel^{\Delta,n}(q) \coloneqq \sum_{m \in \Z} P_{m,\rel}^{\Delta,n} (-q)^m\,.
\end{equation}

It follows from \cite{moop}, \cite[\S 5.1]{mpt} that the Gromov-Witten/stable pairs correspondence is known for equivariant theories with primary insertions of toric 3-folds relatively to smooth toric divisors. 
For $\Delta=\Delta_d^{\PP^2}$ or $\Delta_{h,d}^{\F_k}$, we can state the correspondence as follows (see \cite[Conjecture 3R]{mnop2}):
\begin{itemize}
    \item[(i)] $P_\rel^{\Delta,n}(q)$ is the $q$-Laurent expansion of a rational function in $q$. 
    \item[(ii)]After the change of variables $q=e^{iu}$,
    \begin{equation}  \label{eq_mnop}
        (-i)^{|\Delta|} u^{2h} \sum_g N_{g,\rel}^{\Delta,n,3d} u^{2g-2+d_b+d_t}
        = q^{-\frac{|\Delta|}{2}} P_\rel^{\Delta,n}(q) \,.
    \end{equation}
\end{itemize}
In general, the Gromov-Witten stable pairs correspondence requires to take the logarithm of a generating series of stable pair invariants where we sum over the curve class. However, for  $\Delta=\Delta_d^{\PP^2}$ or $\Delta_{h,d}^{\F_k}$, the curve class is uniquely determined by the tangency conditions, which is why no logarithm or exponential appear in \eqref{eq_mnop}.

Using \eqref{eq_id}, we can combine \eqref{eq:2d_3d} and \eqref{eq_mnop}
to rephrase Theorem \ref{main_thm_precise} as a correspondence between 
$q$-refined counts of floor diagrams and stable pair invariants.

\begin{thm} \label{thm_main_pt}
Let $\Delta$ be a $h$-transverse balanced collection of vectors in 
$\Z^2$ of the form $\Delta^{\PP^2}_d$ or $\Delta^{\F_k}_{h,d}$, and let 
$n$ be a nonnegative integer such that $g_{\Delta,n} \geq 0$. Then we have 
the equality 
\begin{equation} \label{eq_main_pt}
(-1)^{g_{\Delta,n}-1} q^{-\frac{|\Delta|}{2}} P_\rel^{\Delta,n}(q)= t^{g_{\Delta,n}-1}
N^{\Delta,n}_{\floor}(q^{\frac{1}{2}})
\left(
q^{\frac{1}{2}} - q^{-\frac{1}{2}}
\right)^{2 g_{\Delta,n}-2+|\Delta|}\,. 
\end{equation}
\end{thm}

Unlike Theorem \ref{main_thm_precise}, Theorem \ref{thm_main_pt} 
is an equality between rational functions (and in fact Laurent polynomials) in $q^{\frac{1}{2}}$ and no change of variables is required.

\section{Comparison with log invariants} \label{section_log}
Let $\Delta$ be a $h$-transverse balanced collection of vectors in 
$\Z^2$ of the form $\Delta^{\PP^2}_d$ or $\Delta^{\F_k}_{h,d}$,
and let $n$ be a nonnegative integer such that $g_{\Delta,n} \geq 0$.
In \S \ref{section_relative_gw}, we defined Gromov-Witten invariants 
$N^{\Delta,n}_{g,\rel}$ of $X_\Delta$ relative to the smooth divisor 
$D_b \cup D_t$ given by the disjoint union of the two ``horizontal" toric divisors. 
The main result of the present paper, Theorem \ref{main_thm_precise}, 
expresses these relative Gromov-Witten invariants in terms of refined counts of floor diagrams.

In \S 2.2 of \cite{MR3904449}, we defined log Gromov-Witten invariants\footnote{In \cite{MR3904449}, the notation used is 
$N_g^{\Delta,n}$. Here, we use the notation $N_{g, \loga}^{\Delta,n}$
in order to make clear that they are log invariants, a priori distinct from the relative invariants considered in the present paper.} $N_{g,\loga}^{\Delta, n}$ of $X_{\Delta}$ relative to the singular divisor given by the union of the 
toric divisors of $X_{\Delta}$. The difference between $N^{\Delta,n}_{g,\rel}$
and $N_{g,\loga}^{\Delta, n}$ is that in the definition of  $N^{\Delta,n}_{g,\rel}$,
there is no condition involving the non-horizontal toric divisors. In particular, relative stable maps contributing to $N^{\Delta,n}_{g,\rel}$ can have components falling into a non-horizontal toric divisor, whereas such map needs to come with a non-trivial log structure in order to contribute to $N_{g,\loga}^{\Delta, n}$.

The main result of \cite{MR3904449} expresses the log Gromov-Witten invariants $N_{g,\loga}^{\Delta, n}$ in terms of refined counts of tropical curves.
Going back to the original correspondence obtained by Brugallé and Mikhalkin \cite{MR2500574}
between foor diagrams and ``vertically stretched" tropical curves, we obtain an explicit correspondence between $N^{\Delta,n}_{g,\rel}$ and $N_{g,\loga}^{\Delta, n}$. 

\begin{thm} \label{thm_log}
Let $\Delta$ be a $h$-transverse balanced collection of vectors in 
$\Z^2$ of the form $\Delta^{\PP^2}_d$ or $\Delta^{\F_k}_{h,d}$,
and let $n$ be a nonnegative integer such that $g_{\Delta,n} \geq 0$.
Then, $N^{\Delta,n}_{g,\rel}
=  N^{\Delta,n}_{g,\loga}$.
\end{thm}

\begin{proof}
This follows directly from the combination of Theorem
\ref{main_thm_precise}, of Theorem 5 of 
\cite{MR3904449} and from the correspondence between 
floor diagrams and topical curves given by Proposition 5.9 of 
\cite{MR2500574}.
\end{proof}

One can probably obtain a direct proof of Theorem \ref{thm_log}
using a degeneration to the normal cone of the non-horizontal toric 
divisors of $X_\Delta$. One should apply to this degeneration 
an argument in log Gromov-Witten theory, as in 
\cite{MR3904449, bousseau2018quantum_tropical, bousseau2018example}, and use Lemma \ref{lem_vanishing_2}.
Given such direct proof, one could reverse the logic and derive 
Theorem \ref{main_thm_precise} from \cite{MR3904449}, but this would go 
against the spirit of the present paper, which was to remain in the realm of 
relative Gromov-Witten theory and to not use any logarithmic technology.

\section{Application to Block-Göttsche invariants of $\F_0$ and $\F_2$}
\label{section_brugalle_conj}

In this section, as application of Theorem 
\ref{main_thm_precise}, we give a proof of
Conjecture 4.6 of \cite{brugalle2018invariance}, relating Block-Göttsche invariants of $\F_0$ and $\F_2$, see Corollary \ref{cor_AB}.

\subsection{Gromov-Witten invariants of $\F_0$}
We consider $\F_0=\PP^1 \times \PP^1$, and 
we denote $D_0^{\F_0}=\PP^1 \times \{0\}$,
$D_{-0}^{\F_0}=\PP^1 \times \{\infty\}$, $F^{\F_0}=\{0\} \times \PP^1$,
and
$\beta_{a,b}^{\F_0} \coloneqq a[D_0^{\F_0}]
+(a+b) [F^{\F_0}] \in H_2(\F_0,\Z)$.
Using notations of \S \ref{section_h_transverse}, we have 
\begin{equation}\beta_{a,b}^{\F_0}
=\beta_{\Delta^{\F_0}_{a,a+b}} \,.\end{equation}
Applying \eqref{eq_gw_h} to $\Delta=\Delta^{\F_0}_{a,b}$
defines relative Gromov-Witten invariants
\begin{equation} \label{eq_f_0_relative}
    N_{g,\rel}^{\Delta^{\F_0}_{a,a+b},n}
\end{equation}
of $\F_0 /(D_0^{\F_0} \cup D_{-0}^{\F_0})$. In this case, the class 
$\lambda_{g-g_{\Delta,n}}$ is inserted, where 
$g_{\Delta,n}=n+1-|\Delta|=n+1-4a-2b$.

On the other hand, we also consider absolute Gromov--Witten invariants of 
$\F_0$:
\begin{equation}\label{eq_f_0_absolute} N_{g,(a,b)}^{\F_0,n} 
\coloneqq \langle (-1)^{g-(n+1-4a-2b)} \lambda_{g-(n+1-4a-2b)}; \alpha_1,\dots,\alpha_n \rangle_{g,\beta_{a,b}^{\F_0}}^{\F_0} \,,
\end{equation}
where $\alpha_1,\dots,\alpha_n \in H^4(\F_0,\Z)$ are cohomology classes Poincaré dual to a point.
The following Lemma compares the relative invariants 
\eqref{eq_f_0_relative} with the absolute invariants 
\eqref{eq_f_0_absolute}.

\begin{lem} \label{lem_F_0}
For every nonnegative integers $a$, $b$ and $n$ such that $n+1-4a-2b \geq 0$,
\begin{equation}
N_{g,(a,b)}^{\F_0,n}
= N_{g,\rel}^{\Delta^{\F_0}_{a,a+b},n} 
\end{equation}
\end{lem}

\begin{proof}
This follows from the degeneration formula in Gromov-Witten theory
\cite{MR1938113}
applied to the degeneration of $\F_0$ to the normal cone of the smooth divisor $D_0^{\F_0} \cup D_{-0}^{\F_0}$, that is, to $\F_0^{(1)} \cup \F_0^{(2)} \cup \F_0^{(3)}$, keeping the $n$ point insertions in the middle $\F_0^{(2)}$.
By Lemma \ref{lem_gluing1}, the non-zero terms in the degeneration formula are indexed by genus $0$ graphs $\Gamma$ and we insert a top lambda class at every vertex.
If $V$ is a vertex of $\Gamma$ corresponding to a curve in $\F_0^{(1)}$
or $\F_0^{(3)}$ of class $h_V[D_0^{\F_0}]+d_V[F^{\F_0}]$, genus $g_V$, and contact orders defining a partition $\mu$ of $d_V$, the dimension of the corresponding moduli stack of relative stable maps is $g-1+d_V+2h_V+\ell(\mu)$, whereas the maximal degree of an insertion is $g+\ell(\mu)$ (corresponding to the insertion of $\lambda_{g_V}$ and to fixing the position of the $\ell(\mu)$ contact points). 
Therefore, the contribution of $\Gamma$ is $0$ unless 
$g-1+d_V+2h_V+\ell(\mu) \leq g+\ell(\mu)$, that is, $d_V+2h_V \leq 1$, and so $d_V=1$ and $h_V=0$. By Lemma 
\ref{lem_vanishing_2}, the contribution of such vertex is $1$ if $g_V=0$ and $0$ if $g_V>0$.
\end{proof}

For $g=n+1-4a-2b$, Lemma
\ref{lem_F_0} reduces to the well-known fact that absolute and relative Gromov-Witten invariants of $\F_0$ coincide (and are in fact enumerative).

\subsection{Relative Gromov-Witten invariants of $\F_2$}
We denote $D_2^{\F_2}$ and $D_{-2}^{\F_2}$
the toric divisors of $\F_2$ such that 
$(D_2^{\F_2})^2=2$ and $(D_{-2}^{\F_2})^2
=-2$. We denote $[F^{\F_2}]$ the class of a fiber of the natural projection $\F_2 \rightarrow \PP^1$.
Using notations of \S \ref{section_h_transverse}, we have 
\begin{equation} \beta_{\Delta^{\F_2}_{h,d}}
= h[D_2^{\F_2}] + d[F^{\F_2}]\,.\end{equation}
Applying \eqref{eq_gw_h} to $\Delta=\Delta^{\F_2}_{h,d}$
defines relative Gromov-Witten invariants
\begin{equation} \label{eq_f_2_relative}
    N_{g,\rel}^{\Delta^{\F_2}_{h,d},n}
\end{equation}
of $\F_2 /(D_2^{\F_2} \cup D_{-2}^{\F_2})$. 
In this case, the class 
$\lambda_{g-g_{\Delta,n}}$ is inserted, where 
$g_{\Delta,n}=n+1-|\Delta|=n+1-4h-2d$.

On the other hand, we also consider relative Gromov--Witten invariants of 
$\F_2/D_{-2}^{\F_2}$:
\begin{equation}\label{eq_f_2_absolute} N_{g,(h,d)}^{\F_2/D_{-2}^{\F_2},n} 
\coloneqq \langle (-1)^{g-(n+1-4h-2d)} \lambda_{g-(n+1-4h-2d)}; \alpha_1,\dots,\alpha_n |(\eta^1,\delta^1)\rangle_{g,\beta_{\Delta_{h,d}^{\F_2}}}\,,
\end{equation}
where $\alpha_1,\dots,\alpha_n \in H^4(\F_2,\Z)$ are cohomology classes Poincaré dual to a point, $\eta^1=(\eta^1_j)_{1\leq j\leq d}$ is the partition of 
$\beta_{\Delta_{h,d}^{\F_2}} \cdot D_{-2}^{\F_2}$ whose all parts are equal to $1$, and $\delta^1=(\delta^1_j)_{1\leq j\leq d}$ with $\delta^1_j=1 \in H^0(D_{-2}^{\F_2},\Z)$ for all $j$.
The following Lemma compares the relative invariants 
\eqref{eq_f_2_relative} of $\F_2 /(D_2^{\F_2} \cup D_{-2}^{\F_2})$
with the absolute invariants 
\eqref{eq_f_2_absolute} of 
$\F_2/D_{-2}^{\F_2}$.

\begin{lem} \label{lem_F_2}
For every nonnegative integers $h$, $d$  and $n$ such that $n+1-4h-2d \geq 0$, 
\begin{equation}
N_{g,(h,d)}^{\F_2/D_{-2}^{\F_2},n}
=  N_{g,\rel}^{\Delta^{\F_2}_{h,d},n} 
\,.\end{equation}
\end{lem}

\begin{proof}
This follows from the degeneration formula in Gromov-Witten theory
\cite{MR1938113}
applied to the degeneration of $\F_2$ to the normal cone of the smooth divisor $D_2^{\F_2}$, that is, to $\F_2^{(1)} \cup \F_2^{(2)}$, keeping the $n$ point insertions in the first $\F_2^{(2)}$. We argue as in the proof of Lemma \ref{lem_F_0}.
By Lemma \ref{lem_gluing1}, the non-zero terms in the degeneration formula are indexed by genus $0$ graphs $\Gamma$ and we insert a top lambda class at every vertex.
If $V$ is a vertex of $\Gamma$ corresponding to a curve in $\F_2^{(2)}$
of class $h_V[D_2^{\F_2}]+d_V[F^{\F_2}]$, genus $g_V$, and contact orders defining a partition $\mu$ of $d_V$, the dimension of the corresponding moduli stack of relative stable maps is $g-1+d_V+4h_V+\ell(\mu)$, whereas the maximal degree of an insertion is $g+\ell(\mu)$ (corresponding to the insertion of $\lambda_{g_V}$ and to fixing the position of the $\ell(\mu)$ contact points). 
Therefore, the contribution of $\Gamma$ is $0$ unless 
$g-1+d_V+4h_V+\ell(\mu) \leq g+\ell(\mu)$, that is, $d_V+4h_V \leq 1$, and so $d_V=1$ and $h_V=0$. By Lemma 
\ref{lem_vanishing_2}, the contribution of such vertex is $1$ if $g_V=0$ and $0$ if $g_V>0$.
\end{proof}

For $g=n+1-4h-2d$, Lemma
\ref{lem_F_2} reduces to the well-known fact that Gromov-Witten invariants of $\F_2$ relative to 
$D_2^{\F_2} \cup D_{-2}^{\F_2}$ and Gromov-Witten invariants of $\F_2$ relative to 
$D_{-2}^{\F_2}$ coincide (and are in fact enumerative).

\subsection{Comparison of invariants of $\F_0$ and $\F_2$}

\begin{figure}
\center{\scalebox{0.3}{\input{S1.pspdftex}}}
\caption{Degeneration of a curve in the degeneration of 
$\F_0=\PP^1 \times \PP^1$ to the normal cone of the diagonal 
$\Delta \subset \F_0$, with $a=3$, $b=2$, $j=2$.}
\label{Fig:S1}
\end{figure}

\begin{thm} \label{thm_AB}
For every nonnegative integers $a$, $b$ and $n$ such that $n+1-4a-2b \geq 0$, and for every $g  \geq n+1-4a-2b$, 
\begin{equation}
N^{\F_0,n}_{g,(a,b)}= \sum_{j=0}^a
\left( \begin{array}{c}
b+2j  \\
j  \end{array} \right)
N^{\F_2/D_{-2}^{\F_2},n}_{g,(a-j,b+2j)} \,.
\end{equation}
\end{thm}

\begin{proof}
This follows from the degeneration formula in Gromov-Witten theory \cite{MR1938113} applied to
the degeneration of $\F_0=\PP^1 \times \PP^1$ to the normal cone of its diagonal
$\Delta$. As $\Delta \cdot \Delta=2$, the special fiber is $\F_0 \cup \F_2$, and we send the
$n$ point insertions to $\F_2$.

By Lemma \ref{lem_gluing1}, the non-zero terms in the degeneration formula are indexed by genus $0$ graphs $\Gamma$ and we insert a top lambda class at every vertex.
If $V$ is a vertex of $\Gamma$ corresponding to a curve in $\F_0$
of class $\beta_V=h_V[D_0^{\F_0}]+d_V[F^{\F_0}]$, genus $g_V$, and contact orders defining a partition $\mu$ of $\beta_V \cdot \Delta = h_V+d_V$, the dimension of the corresponding moduli stack of relative stable maps is $g-1+d_V+h_V+\ell(\mu)$, whereas the maximal degree of an insertion is $g+\ell(\mu)$ (corresponding to the insertion of $\lambda_{g_V}$ and to fixing the position of the $\ell(\mu)$ contact points). 
Therefore, the contribution of $\Gamma$ is $0$ unless 
$g-1+d_V+h_V+\ell(\mu) \leq g+\ell(\mu)$, that is, $d_V+h_V \leq 1$, and so either $(d_V,h_V)=(1,0)$ or $(d_V,h_V)=(0,1)$. By Lemma 
\ref{lem_vanishing_2}, the contribution of such vertex is $1$ if $g_V=0$ and $0$ if $g_V>0$. So, we assume that all vertices $V$ in 
$\Gamma$ corresponding to curves in $\F_0$ have $g_V=0$, and either 
$(h_V,d_V)=(1,0)$ or $(h_V,d_V)=(0,1)$.
We denote by $j$ the number of such vertices with $(h_V,d_V)=(1,0)$.
As $\beta_{a,b}^{\F_0} \cdot F^{\F_0}=a$, we have 
$1 \leq j \leq a$.

As $\Gamma$ is connected, there exists a unique vertex $V_0$ of 
$\Gamma$ corresponding to curves in $\F_2$, and the associated genus is 
$g_{V_0}=g$.
As $\beta_{a,b}^{\F_0} \cdot \Delta = 2a+b$, the associated class $\beta_{V_0}$
satisfies $\beta \cdot D_2^{\F_2}=2a+b$ and 
$\beta \cdot F^{\F_2}=a-j$. This uniquely determines $\beta_{V_0}$
to be $(a-j)[D_2^{\F_2}]+(b+2j)[F^{\F_2}]=\beta_{\Delta^{\F_2}_{a-j,b+2j}}$. In particular,
we have $\beta_{V_0} \cdot D_{-2}^{\F^2}=b+2j$, and so there are in total 
$b+2j$ vertices in $\Gamma$ corresponding to curves in $\F_0$.
The binomial coefficient
\begin{equation} \left( \begin{array}{c}
b+2j  \\
j  \end{array} \right)\end{equation}
is the number of ways to 
choose the $j$ vertices with $(h_V,d_V)=(1,0)$
and the $b+j$ vertices with $(h_V,d_V)=(0,1)$ among the 
$b+2j$ vertices corresponding to curves in $\F_0$, see Figure \ref{Fig:S1}.
\end{proof}

For $g=n+1-4a-2b$,
Theorem \ref{thm_AB} reduces to the classical formula, due to Abramovich-Bertram \cite{MR1837110} in genus zero and to Vakil \cite{MR1771228} in higher genus,
comparing enumerative invariants of $\F_0$ and $\F_2$. 

Finally, the following corollary is the precise version of
Theorem \ref{thm_f0_f2_intro}.

\begin{cor} \label{cor_AB}
For every nonnegative integers $a$, $b$, and $n$ such that $n+1-4a-2b \geq 0$,
we have 
\begin{equation} N_\floor^{\Delta^{\F_0}_{a,a+b},n}(q^{\frac{1}{2}})
= \sum_{j=0}^a 
\left( \begin{array}{c}
b+2j  \\
j  \end{array} \right)
N_\floor^{\Delta^{\F_2}_{a-j,b+2j},n}(q^{\frac{1}{2}})\,.
\end{equation}
\end{cor}

\begin{proof}
This follows from the combination of Theorem \ref{thm_AB},
Theorem \ref{main_thm_precise}, Lemma
\ref{lem_F_0} and Lemma \ref{lem_F_2}.
The only thing to check is the cancellation of the factors $u$ and 
$(-i)(q^{\frac{1}{2}}-q^{-\frac{1}{2}})$.
\end{proof}

 The statement of Corollary 
\ref{cor_AB} is Conjecture 4.6 of \cite{brugalle2018invariance}. The previously known cases of this Conjecture were the specialization $q=1$ (that is, the formula proved by Abramovich-Bertram and Vakil), and for $b=0$, $n=4a-1$, the specialization $q=-1$ (as consequence of a surgery formula for Welschinger invariants, see Proposition 2.7 of \cite{MR3433282}).

It is implicit in  \cite{brugalle2018invariance} and motivated by a surgery formula for Welschinger invariants with pairs of complex conjugated  point constraints,
    that a version of Conjecture 4.6 of   \cite{brugalle2018invariance} should also hold for a class of G\"ottsche-Schroeter invariants, \cite{MR3875360}, tropical refinement of genus zero Gromov-Witten counts for which some of the point insertions also come with insertion of a psi class.
    This conjecture can be proved as Corollary \ref{cor_AB} using the geometric interpretation of 
    G\"ottsche-Schroeter invariants given in Appendix B of the first arxiv version of \cite{MR3904449}. It does not seem completely obvious to obtain a proof in the spirit of the present paper, that is, without tropical and logarithmic technology.

Corollary \ref{cor_AB} is an equality between 
combinatorially defined $q$-refined counts of floor diagrams, and so, as suggested in \cite{brugalle2018invariance}, it is very likely that a combinatorial proof exists. Our proof is geometric: once we have, thanks to Theorem \ref{main_thm_precise}, a Gromov-Witten interpretation of these combinatorial 
objects, we just used the fact that the usual proof by degeneration of the Abramovich-Bertram formula 
goes through and gives the $q$-refined statement.

\vspace{+8 pt}
\noindent
Institute for Theoretical Studies \\
ETH Zurich \\
8092 Zurich, Switzerland \\
pboussea@ethz.ch

\end{document}